\documentclass[11pt]{amsart}

\usepackage{amsmath}
\usepackage{amssymb}
\usepackage{graphicx}
\newtheorem{theorem}{Theorem}[section]
\newtheorem{corollary}[theorem]{Corollary}
\newtheorem{lemma}[theorem]{Lemma}
\newtheorem{proposition}[theorem]{Proposition}
\theoremstyle{definition}
\newtheorem{definition}[theorem]{Definition}

\newtheorem{remark}[theorem]{Remark}

\numberwithin{equation}{section}
\usepackage{mathtools}
\usepackage{enumerate}
\mathtoolsset{showonlyrefs}

\makeatletter
\newtheorem*{rep@theorem}{\rep@title}
\newcommand{\newreptheorem}[2]{%
\newenvironment{rep#1}[1]{%
 \def\rep@title{#2 \ref{##1}}%
 \begin{rep@theorem}}%
 {\end{rep@theorem}}}
\makeatother
\newreptheorem{theorem}{Theorem}
\newreptheorem{corollary}{Corollary}

\newcommand{\R}{\mathbb{R}}

\newcommand{\N}{\mathbb{N}}

\newcommand{\aff}[1]{\ell_{#1}}

\newcommand{\s}{\mathbb{S}^{n-1}}
\newcommand{\supp}{\text{supp}}
\newcommand{\conbod}{\mathcal{K}^n}

\newcommand{\vol}{\text{\rm Vol}}
\def \Rn{{\R^n}}
\newcommand{\hid}{m}
\def \Rnhi{{\R^{n\hid}}}

\def \P{\Pi^{\hid}}
\def \PP{\Pi^{\circ,\hid}}

\def \Cov{g_{K,\hid}}
\def \S{\mathbb{S}^{n\hid-1}}

\title[The $m\mathrm{th}$-Order Weighted Projection Body]{On the $m\mathrm{th}$-Order Weighted Projection Body \\ Operator and Related Inequalities}

\author[D. Langharst]{Dylan Langharst}

\address[D. Langharst]{Institut de Math\'ematiques de Jussieu, Sorbonne Universit\'e, Paris, 75252, France}
\email{{\tt dylan.langharst@imj-prg.fr}}

\author[E. Putterman]{Eli Putterman}
\address[E. Putterman]{School of Mathematical Sciences, Tel Aviv University, Tel Aviv 66978, Israel}
\email{\tt putterman@mail.tau.ac.il}

\author[M. Roysdon]{M. Roysdon}

\address[M. Roysdon]{Department of Mathematics, Applied Mathematics, and Statistics, Case Western Reserve, Cleveland, OH 44106, USA}
\email{{\tt mar327@case.edu}}

\author[D. Ye]{Deping Ye}
\address[D. Ye]{Department of Mathematics, and Statistics, Memorial University of Newfoundland, St. John’s, Newfoundland, A1C 5S7, Canada}
\email{\tt deping.ye@mun.ca}



\begin{document}

\begin{abstract}
For a convex body $K$ in $\R^n$, the inequalities of Rogers-Shephard and Zhang, written succinctly, are $$\vol_n(DK)\leq \binom{2n}{n} \vol_n(K) \leq \vol_n(n\vol_n(K)\Pi^\circ K).$$ Here, $DK=\{x\in\R^n:K\cap(K+x)\neq \emptyset\}$ is the difference body of $K$, and $\Pi^\circ K$ is the polar projection body of $K$. There is equality in either if, and only if, $K$ is a $n$-dimensional simplex. In fact, there exists a collection of convex bodies, the so-called radial mean bodies $R_p K$ introduced by Gardner and Zhang, which continuously interpolates between $DK$ and $\Pi^\circ K$. For $m\in\N$, Schneider defined the $m$th-order difference body of $K$ as $$D^m(K)=\{(x_1,\dots,x_m)\in\R^{nm}:K\cap_{i=1}^m(K+x_i)\neq \emptyset\}\subset \R^{nm}$$ and proved the $m$th-order Rogers-Shephard inequality. In a prequel to this work, the authors, working with Haddad, extended this $m$th-order concept to the radial mean bodies and the polar projection body, establishing the associated Zhang's projection inequality. 
 
 In this work, we introduce weighted versions of the above-mentioned operators by replacing the Lebesgue measure with measures that have density. The weighted version of these operators in the $m=1$ case was first done by Roysdon (difference body), Langharst-Roysdon-Zvavitch (polar projection body) and Langharst-Putterman (radial mean bodies). This work can be seen as a sequel to all those works, extending them to $m$th-order. In the last section, we extend many of these ideas to the setting of generalized volume, first introduced by Gardner-Hug-Weil-Xing-Ye.
\end{abstract}

\maketitle

%%%%%%%%%%%%%%%%%%%%%%%%%%%%%%%%%%%%%%%%%%%%%%%%%%%%%%%%%%%%%%%%%%%%%%%%%%%%%%%%

\section{Introduction}
\label{sec:intro}
Let $\R^n$ be the standard $n$-dimensional Euclidean space and  $\s$ be the unit sphere in $\mathbb{R}^n$.  By  $\conbod$ we mean the set of convex bodies (compact, convex sets with non-empty interiors) in $\R^n$.  Let $\conbod_0$ be the set of those convex bodies containing the origin $o$ in their interiors. The support function of $K \in \conbod$, denoted by $h_K:\R^n \to \R$ is defined by  $h_K(x)=\sup \{\langle x,y\rangle \colon y \in K \}$, where $\langle x, y\rangle$ denotes the inner product of $x, y\in \Rn$.  Let $\theta^{\perp}=\{x\in\R^n:\langle \theta,x\rangle =0\}$ be the subspace orthogonal to $\theta \in \s$. By $P_H K$ we mean the orthogonal projection of $K$ onto a linear subspace $H$. It is well known that the function $\theta\in \s \mapsto \vol_{n-1}(P_{\theta^{\perp}}K)$ is a sublinear functional and hence defines a convex body in $\mathbb{R}^n$;  such a convex body is called the \textbf{projection body} of $K$ and denoted by $\Pi K$. Hereafter,  $\vol_m$ refers to the $m$-dimensional Lebesgue measure (volume). From \eqref{eq:cauch} below, one can easily see that $\Pi K$ is an origin symmetric convex body, where a convex body $K$ is said to be origin symmetric if $K=-K$. The projection body $\Pi K$ plays important roles in affine isoperimetric inequalities including, e.g., \textbf{Zhang's projection inequality} \cite{Zhang91} and \textbf{Petty's projection inequality} \cite{CMP71}.  These two inequalities can be stated as follows: for any $K\in\conbod$, one has
\begin{equation}\label{e:Zhang_ineq}
\frac{1}{n^n} \binom{2n}{n}\leq\vol_n(K)^{n-1}\vol_n(\Pi^\circ K)\leq \left(\frac{\omega_n}{\omega_{n-1}}\right)^n,
\end{equation} where  $\omega_n$ denotes the volume of the unit ball $B_2^n$ in $\R^n$ and $$K^\circ=\left\{x\in \R^n: h_K(x) \leq 1\right\}$$ defines the polar body of $K \in\conbod_0$.  Equality holds for the first inequality in \eqref{e:Zhang_ineq} (i.e., Zhang's inequality) if, and only if, $K$ is a $n$-dimensional simplex (convex hull of $n+1$ affinely independent points); equality holds for  the second inequality in \eqref{e:Zhang_ineq} (i.e.,  Petty's inequality) if, and only if, $K$ is an ellipsoid. The proof of  Petty's inequality follows from the classical Steiner symmetrization, and the proof of Zhang's inequality, as presented in \cite{GZ98},  made critical use of the covariogram function.

The covariogram function is a fundamental notion in convex geometry and appears often in the literature. We denote by $g_K: \R^n\rightarrow \R$, the covariogram function of $K\in\conbod$, which is given by (see e.g., \cite{GB23}) 
	\begin{equation}\label{e:covario}
	    g_K(x)=\vol_n\left(K\cap(K+x)\right)=(\chi_K\star\chi_{-K})(x)\ \ \ \mathrm{for}\ x\in \R^n, 
	\end{equation}
where $(f\star g)(x)= \int_{\R^n} f(y)g(x-y)\,dy$ is a convolution of functions $f,g:\R^n \to \R$ and
	$\chi_K(x)$ is the characteristic function of $K$. 
The significance of the covariogram function can be seen immediately from the following facts. First of all, it is easily checked that the support of $g_K(x)$ is the difference body of $K$ given by 
	\begin{equation}\label{e:differnce}
	    DK=\{x:K\cap(K+x)\neq \emptyset\}=K+(-K),
	\end{equation} where the Minkowski sum of two Borel sets is given by $K+L=\{ x + y: x\in K, y\in L\}.$ 
 
 Among those important results related to the difference body are the following inequalities: for $K\in\conbod$ one has
	\begin{equation}\label{e:roger_shep}
	2^n\leq\frac{\vol_n(DK)}{\vol_n(K)}\leq \binom{2n}{n}. 
	\end{equation} The first inequality is a direct consequence of the Brunn-Minkowski inequality (see e.g. \cite{G20}), and equality holds if, and only if, $K$ is symmetric about a point (i.e. a translate of $K$ is origin symmetric). This provides a measurement for the symmetry of convex bodies. The second inequality is the famous  \textbf{Rogers-Shephard inequality} \cite{RS57}, in which equality holds if, and only if, $K$ is a $n$-dimensional simplex. Chakerian \cite{Cha67} was the first to show the connection between the covariogram and the difference body, yielding a succinct proof of the Rogers-Shephard inequality. The proof uses another property of the covariogram function: it inherits the $1/n$ concavity property of the Lebesgue measure, because the covariogram function is defined as a convolution of characteristic functions of convex bodies. 
 
 Secondly, the classical result by Matheron \cite{MA} shows  that $\frac{dg_{K}(r\theta)}{dr}\big|_{r=0^+}=-h_{\Pi K}(\theta).$ Thus, the covariogram function naturally leads to the projection body via the variational approach, which is one of the crucial steps in the proof of Zhang's projection inequality in \cite{GZ98}. In particular, Gardner and Zhang showed that 
\begin{equation}
    \label{e:set_inclusion_og}
    D K \subseteq n \vol_n(K) \Pi^{\circ} K,
\end{equation}
with equality if, and only if, $K$ is a $n$-dimensional simplex. %Eli: "in particular?"

Formula \eqref{e:covario} for the covariogram function $g_K$ can be adapted to define the $m$th-order covariogram function in a natural way. To do this, we always identify, throughout this paper,  $\R^{n\hid}$ with the product structure $ \R^n \times \cdots \times \R^n$, and a point in $\Rnhi$ is often written as $\bar{x}=(x_1,\dots,x_\hid)$, with each $x_i\in\Rn$. Note that the product structure 
is of central importance in later context. Schneider \cite{Sch70} proposed the $\hid$\textit{th-order covariogram function} $\Cov: \R^{mn}\rightarrow \R$ for $K\in\conbod$ and $\hid\in\mathbb{N}$ as follows:  
\begin{equation}
\Cov(\overline{x})=\vol_n\left(K\cap_{i=1}^\hid (x_i+K)\right) \ \ \ \mathrm{for}\ \bar{x}\in \R^{mn}.
\label{eq_vol_ell_co}
\end{equation} When $m=1$, one obtains the covariogram function $g_K$ in \eqref{e:covario}. Consequently, one can define the $\hid$\textit{th-order difference body of K}, denoted by $D^\hid(K)$, as the support of the $\hid$th-order covariogram function $\Cov$.  Clearly $D^\hid(K) $ is a convex body in $(\R^{n})^\hid=\R^{n\hid}$ and
\[
D^\hid(K) := \{\overline{x}=(x_1,\dots,x_{\hid}) \in (\R^n)^{\hid} \colon K \cap_{i=1}^\hid (x_i+K) \neq \emptyset \}.
\]  An important inequality related to  
$D^\hid(K)$ is the $m$th-order Rogers-Shephard inequality proved by Schneider  \cite{Sch70}:
\begin{equation}
    \vol_n(K)^{-\hid}\vol_{n\hid}\left(D^\hid (K)\right)\leq \binom{n\hid +n}{n}
    \label{eq:RSell},
\end{equation}
with equality if, and only if, $K$ is a $n$-dimensional simplex. A lower bound for the left-hand side of \eqref{eq:RSell}, which can be shown to be nonzero due to affine invariance, is obtained for arbitrary $\hid\in\mathbb{N}$ and for every symmetric convex body $K$ when $n=2$, and is conjectured to be attained by ellipsoids when $n\geq 3$ and $\hid\geq 2$. 

Through a variational approach, the $\hid$th-order covariogram function $\Cov$ also leads to the $m$th-order projection body of $K$. In \cite{HLPRY23}, together with Haddad, the authors proved that  \begin{equation}
    \frac{d}{dr}\Cov(r\overline{\theta}) \bigg|_{r=0^+}=-h_{\Pi^\hid K}(\overline{\theta}) \ \ \ \mathrm{for} \ \bar{\theta}\in \S,
    \label{eq_vol_ell_co_deriv}
\end{equation}
where $\Pi^\hid K$, the $\hid$\textit{th-order projection body} of $K\in \conbod$, is the $n\hid$-dimensional convex body whose support function is given by
$$h_{\Pi^\hid K}(\overline{x})=
h_{\Pi^{\hid}K}(x_1,\dots,x_{\hid}) = \int_{\partial K} \max_{1 \leq i \leq \hid} \langle x_i,n_K(y) \rangle_- dy.
$$
In the above formula, $\partial K$ denotes the topological boundary of $K$, $dy=d\mathcal{H}^{n-1}(y)$, $\mathcal{H}^{n-1}$ is the $(n-1)$-dimensional Hausdorff measure on $\partial K$, $n_K:\partial K \rightarrow \s$ is the Gauss map, which associates $y\in \partial K$ with its outer unit normal, and, for $a \in \mathbb R$, $a_- = -\min(a, 0)$. The sharp upper and lower bounds for $\vol_n(K)^{n\hid-\hid}\vol_{n\hid}\left( \Pi^{\circ,\hid} K \right)$, with  $\Pi^{\circ,\hid} K= (\Pi^{\hid} K)^{\circ}$, have been established in \cite{HLPRY23}. As an example, we only mention the lower bound, i.e, the $m$th-order Zhang's projection inequality  \cite{HLPRY23}: \begin{theorem}\label{t:ZhangInequalityVol}
Fix $\hid \in \N$ and $K\in\conbod.$ Then, one has 
\[
\vol_n(K)^{n\hid-\hid}\vol_{n\hid}\left( \Pi^{\circ,\hid} K \right) \geq \frac{1}{n^{n\hid}}\binom{n\hid+n}{n}.
\] Equality occurs if, and only if, $K$ is a $n$-dimensional simplex. The latter is also equivalent to the condition $D^\hid(K)= n\vol_n(K)\Pi^{\circ,\hid}K.$
\end{theorem} Note that the $m$th-order Zhang's projection inequality was proved in  \cite{HLPRY23}  by using the following chain of set inclusions: for $K\in\conbod$, $\hid\in\mathbb{N}$, and $-1< p\leq q < \infty$, one has
$$D^\hid (K) \subseteq {\binom{q+n}{n}}^{\frac{1}{q}} R^\hid_{q}K \subseteq {\binom{p+n}{n}}^{\frac{1}{p}} R^\hid_{p} (K)\subseteq n\vol_n(K)\PP K,$$ with equality in any set inclusion if, and only if, $K$ is an $n$-dimensional simplex. Here $R^\hid_{p} (K)$ is the $(\hid, p)$ radial mean body of $K$ (i.e., the special case of Definition \ref{definition-m-radial mean body} where $\mu$ is the Lebesgue measure in $\R^n$). Note that the above chain of set inclusions provides an alternative proof for the inequality \eqref{eq:RSell} proven by Schneider. When $m=1$, the chain of set inclusions reduces to set inclusions for the {\it radial mean bodies} $R_p K$ shown by Gardner and Zhang \cite{GZ98}. 

The aforementioned $m$th-order results are based on the Lebesgue measure on $\R^n.$ A natural question to ask is whether these results can be extended to general measures on $\R^n.$ A major goal of this paper is to study the $m$th-order weighted projection body and prove a related Zhang's projection inequality. Before we do this, we now introduce some basic notions. 

 For convenience, let  $\lambda_n$ denote the Lebesgue measure on $\Rn$; we will simply write $\lambda$ when possible. We say a measure has density if it has density with respect to $\lambda$. That is, a measure has density if
\[
\frac{d\mu}{d\lambda} = \phi, \ \ \text{ with }\  \phi \colon \R^n \to \R^+,\ \ \phi\in L^1_{\text{loc}}(\R^n). 
\]
All such measures are examples of \textit{Radon measures}, i.e., locally finite and regular Borel measures on $\R^n$. For a fixed $K\in\conbod$, we denote by $\mathcal{M}(K)$ the class of Radon measures $\mu$ on $\Rn$ such that $\mu$ has locally integrable density $\phi \colon \R^n \to \R^+$, and additionally contains $K$ in its support and $\partial K$ in its Lebesgue set. For $\mu\in\mathcal{M}(K)$, define the Borel measure $S^\mu_K$ on $\s$, called the \textit{weighted surface area measure}, by the following formula: 
$$\int_{\s}f(u) \, dS^{\mu}_K(u) = \int_{\partial K} f(n_K(y))\phi(y)\,dy$$ for $f\in {C}(\s),$ where  ${C}(\s)$ denotes the set of all continuous functions on $\s.$ Here, $\phi$ is understood on $\partial K$ as the representative of $\phi$ given by the Lebesgue differentiation theorem. A more ``hands-on" definition will be given in Definition~\ref{def:surface_mu}; note that the measure $S^{\mu}_K$ satisfies \eqref{eq:surface_mu} in that definition. It can be checked that  
\begin{equation}\label{eq_bd}S^{\mu}_K(\mathbb{S}^{n-1})=
    \mu^+(\partial K) := \liminf_{\epsilon\to 0}\frac{\mu\left(K+\epsilon B_2^n\right)-\mu(K)}{\epsilon} = \int_{\partial K}\phi(y)\,dy.
\end{equation} 
 The \textit{weighted projection body} of a convex body $K\in \conbod$ \cite{LRZ22}, denoted $\Pi_\mu K,$ can be defined via its support function: 
\begin{equation}
\begin{split}
    h_{\Pi_{\mu}K}(\theta)= \int_{\partial K} \langle n_K(y),\theta\rangle_{-} \phi(y)\, d y \ \ \mathrm{for} \ \ \theta\in \s.
    \end{split}
    \label{e:mu_polar_suppot}
\end{equation} Note that $h_{\Pi_{\mu}K}$ is clearly positive on $\s$ if $\phi$ is positive almost-everywhere on $\partial K$, and hence $\Pi_\mu K\in\conbod_0$. As the case for the Lebesgue measure on $\R^n,$ $\Pi_\mu K$ can be linked with a weighted covariogram function by a variational approach. 
Let $K\in\conbod$ and $\mu$ be a Borel measure on $\Rn$. The $\mu$-\textit{covariogram} of $K$ is the function given by
\begin{equation} 
g_{\mu,K}(x)=\mu(K\cap(K+x)).
\label{covario}
\end{equation} 
In \cite{LRZ22}, the radial derivative of $g_{\mu,K}$ at the origin was computed under some mild conditions, which were later dropped in \cite{LP23}: 
\begin{equation}\label{e:deriv_g_mu_covario}
    \frac{dg_{\mu,K}(r\theta)}{dr}\bigg|_{r=0^+}=-h_{\Pi_{\mu}K}(\theta),
    \end{equation}
when $\mu\in\mathcal{M}(K)$. Combining \eqref{eq_vol_ell_co} and \eqref{covario}, one can define the $m$th-order $\mu$-\textit{covariogram} of $K$ as follows:  

\begin{definition} \label{eq_higher_co} Let $n,\hid \in \N$, $K\in\conbod$, and $\mu$ be a Borel measure on $\R^n$ containing $K$ in its support. The $m$th-order $\mu$-\textit{covariogram} of $K$ is the function $g_{\mu,\hid}(K,\cdot) \colon \R^{n\hid} \to \R^+$ given by 
\[
g_{\mu,\hid}(K, \bar x) = g_{\mu,\hid}(K,(x_1,\dots,x_{\hid})) = \int_{K} \left(\prod_{i=1}^{\hid} \chi_{K}(y-x_i)\right)\,d\mu(y).
\]
\end{definition} Clearly, $D^\hid(K)$ is the support of $g_{\mu,\hid}(K,\cdot)$, and 

\begin{equation}
g_{\mu,\hid}(K, \bar x) =\mu\left(K\cap_{i=1}^\hid (x_i+K)\right) \ \ \ \mathrm{for}\ \bar{x}\in \R^{mn}.
\label{eq_vol_ell_co-mu}
\end{equation}
 We follow in the footsteps of \cite{HLPRY23} and begin by taking the radial derivative of $g_{\mu,\hid}(K,\cdot).$

\begin{theorem} \label{t:variationalformula}
Let $K\in\conbod$, $\hid \in \N$, and $\phi$ be the density of a Borel measure $\mu\in\mathcal{M}(K)$. For every direction $\bar\theta = (\theta_1,\dots,\theta_{\hid}) \in \S$, one has 
\[
\frac{d}{dr} \left[g_{\mu,\hid}(K,r\bar\theta) \right]\bigg|_{r=0^+}
= -\int_{\partial K} \max_{1 \leq i \leq \hid} \langle \theta_i, n_K(y) \rangle_{-} \phi(y)\,dy.\] 
\end{theorem}

 Motivated by Theorem~\ref{t:variationalformula}, we give the following definition for the $\mu$-weighted $\hid$th-order projection body of $K$.

\begin{definition} 
\label{def:meas_bodies}
Given any $K \in \conbod$, $\hid\in\N$, and a measure $\mu\in\mathcal{M}(K)$, we define the $\mu$- weighted $\hid$th-order projection body of $K$ to be the $n\hid$-dimensional convex body $\P_\mu K$ with support function
\[
h_{\P_\mu K}(x_1,\dots,x_{\hid}) = \int_{\s} \max_{1 \leq i \leq \hid} \langle x_i,u \rangle_- dS^{\mu}_K(u), \quad x_i\in\R^n, i=1,\dots,m. 
\]
\end{definition}

With this definition, we can rephrase the result of Theorem~\ref{t:variationalformula} in the following way: let $\mu\in\mathcal{M}(K)$ for a fixed $K\in\conbod$. Then, 
$$
\frac{d}{dr} \left[g_{\mu,\hid}(K,r\bar\theta) \right]\bigg|_{r=0^+}=-h_{\P_\mu K}(\bar\theta).$$
 When comparing the above formulation of Theorem~\ref{t:variationalformula} with \eqref{e:deriv_g_mu_covario} we indeed obtain $h_{\Pi_{\mu}^1 K}(\theta_1)=h_{\Pi_{\mu}K}(\theta_1)$; also, we have $h_{\Pi_{\lambda}^{\hid}K}(\bar\theta)=h_{\P K}(\bar\theta).$ We set $\PP_\mu K:=(\P_\mu K)^\circ$. Using $g_{\mu,\hid}(K,\cdot)$, we define in Definition 
\ref{definition-m-radial mean body} the $\mu$-weighted, $m$th-order radial mean bodies and prove the following theorem. This theorem, which is stated for $s$-concave measures, $s>0$, is merely a special case of Theorem~\ref{t:radial_F_inclusions}.

\begin{theorem}
\label{t:set_s_con}
Let $K\in\conbod$ and $\hid\in\N$ be fixed. Suppose that $\mu$ is an $s$-concave Borel measure, $s>0,$ on convex subsets of $K$. Then, for $-1< p\leq q < \infty$, one has
$$D^\hid (K)\subseteq \binom{\frac{1}{s}+q} {q}^{\frac{1}{q}} R^\hid_{q,\mu}K \subseteq \binom{\frac{1}{s}+p} {p}^{\frac{1}{p}} R^\hid_{p,\mu} K\subseteq \frac{1}{s}\mu(K)\PP_{\mu}K,$$
where the last inclusion holds if $\mu\in\mathcal{M}(K)$.

\vskip 2mm  There is an equality in any set inclusion if, and only if, $$g_{\mu,\hid}^s(K,x)=\mu(K)^s\ell_{D^\hid (K)}(x),$$ where $\ell_K$ is the \textit{roof function} of $K$, defined in \eqref{eq:lk} below. If $\mu$ is a locally finite and regular Borel measure, i.e., $s$-concave on compact subsets of its support, then $s\in (0,1/n]$ and equality occurs if, and only if, $K$ is a $n$-dimensional simplex, $\mu$ is a positive multiple of the Lebesgue measure, and $s = \frac{1}{n}$.
\end{theorem}

Our main results below, Theorem~\ref{t:radial_F_inclusions} and Theorem~\ref{t:set_inclu_log}, require the measure to have concavity of some sort. It is natural to ask if variants of these inequalities can hold for measures without any concavity assumptions. While we cannot definitively say that this is not possible, we are skeptical of such a possibility; such a method would also require providing a proof of the radial mean body set inclusions (and of the Rogers-Shephard inequality and the Zhang's projection inequality) in the case of the Lebesgue measure \textit{without} using the concavity of the Lebesgue measure. Even among $F$-concave measures, it is very unlikely to obtain a result as elegant as Theorem~\ref{t:set_s_con} with concise equality conditions. In this general framework, $s$-concave measures are essentially the only measures we have a complete characterization of. Progress beyond this classification includes the Ehrhard inequality \cite{EHR1,EHR2} for the Gaussian measure, and such a concavity is a special case of the concavity from Theorem~\ref{t:set_inclu_log}. There is also the resolution of the Gardner-Zvavitch conjecture, the fact that the Gaussian measure is $1/n$-concave on the class of origin symmetric convex bodies \cite{GZ10,KL21,EM21}; this was later extended to all rotational invariant, log-concave measures \cite{CER23}. We leave to the reader to verify that, formally, the theorems in this work are compatible with this concavity if the additional assumption of symmetry about the origin is made on the convex body $K$ and the covariogram function $g_{\mu,\hid}(K,\cdot)$ is replaced by the $m$th-order \textit{polarized covariogram}, which can be defined as $r_{\mu,\hid}(K,\bar x)=\mu\left((K-\frac{x_i}{2})\cap_{i=1}^\hid(K+\frac{x_i}{2})\right)$. As shown when $\hid=1$ in \cite{LRZ22,LP23}, this covariogram will inherit any concavity of the measure $\mu$ over the class of origin symmetric convex bodies. Then, the proofs for the relevant versions of the main theorems (in particular Theorem~\ref{t:set_s_con} with $s=1/n$) are the same line-by-line, after replacing $g_{\mu,\hid}(K,\cdot)$ in the appropriate definitions (one would also have to change, for example, the definition of the projection body when proving the variant of Theorem~\ref{t:variationalformula}.)

The paper is organized as follows. In Section~\ref{sec_summary}, we will state some facts from convex geometry that will be used in this paper. Theorem \ref{t:ZhangInequalityVol}, and also equations \eqref{eq_vol_ell_co} and \eqref{eq_vol_ell_co_deriv} are special cases of more general results presented later in this work. In particular, Section~\ref{sec:higher_dim} explores the weighted, $m$th-order analogues of projection bodies and proves the associated Zhang's projection inequality. In Section~\ref{sec:new_rad}, we study $m$th-order, weighted radial mean bodies and prove Theorem~\ref{t:set_s_con}. Finally, in Section~\ref{sec:genvol}, we extend some of our machinery to the notion of generalized volume introduced in \cite{GHX19}.

\section{Preliminaries}
\label{sec_summary}
  Denote by $S_K(\cdot)$ the surface area measure of a convex body $K\in\conbod$. $S_K$ is a Borel measure on $\s$, defined via the formula
  \begin{align} \label{equation-surface area measure} 
S_K(A)=\mathcal{H}^{n-1}(n^{-1}_K(A)) \ \ \mathrm{for\ every\ Borel\ set } \ A \subset \s, \end{align}  
  where $n_K$ is the Gauss map of $K$. Recall that, for $y \in \partial K$, $n_K(y)$ is the outer unit normal vector to $K$ at $y$, which is uniquely defined for almost all $y \in \partial K$. As the set $N_K=\{x\in\partial K: \text{$K$ does not have a unique outer normal at $x$}\}$ is of $\mathcal H^{n - 1}$-measure zero, we continue to write $\partial K$ in place of $\partial K\setminus N_K,$ without any confusion. The \textit{Cauchy projection formula}  (see e.g., \cite[pg. 408, Eq. (A.45)]{gardner_book}) states that: for  $\theta\in\mathbb{S}^{n-1},$  
	\begin{equation}
        \label{eq:cauch}
	   \vol_{n-1}\left(P_{\theta^{\perp}}K\right) =\frac{1}{2}\int_{\s}|\langle \theta,u \rangle| dS_K(u)=\int_{\s}\langle \theta,u \rangle_- dS_K(u),
	\end{equation}
 where the last equality follows from the fact that $S_K$ has center of mass at the origin.
 
Following \eqref{equation-surface area measure}, we define the weighted surface area measure as follows. Recall that if $f: X \to Y$ is a measurable map between measurable spaces and $\mu$ is a measure on $X$, then the pushforward of $\mu$ by $f$, denoted $f_* \mu$, is the measure on $Y$ defined by $(f_* \mu)(A) = \mu(f^{-1}(A))$.

\begin{definition}
\label{def:surface_mu}
	For a convex body $K\in\conbod$ and a Borel measure $M$ on $\partial K$ which is absolutely continuous with respect to the Hausdorff measure on $\partial K$, the $M$-surface area measure of $K$ is defined as $S^{M}_K= (n_{K})_* M$, the pushforward of $M$ by $n_K$. 
 
 Writing $\phi$ for the density of $M$ with respect to $\mathcal{H}^{n - 1}$, we have
\begin{equation}
    \label{eq:surface_mu}
    S^{M}_{K}(A)=M(n_K^{-1}(A))=\int_{n_K^{-1}(A)}\phi(x)\,d\mathcal{H}^{n-1}(x)
\end{equation}
for every Borel set $A \subset \s.$ 
\end{definition} 

In particular, for $\mu\in\mathcal{M}(K)$ with density $\phi$, the measure $\phi(x)\,d\mathcal H^{n - 1}(x)$ can be used to define a surface area measure via \eqref{eq:surface_mu} which is the same as in \eqref{eq_bd}.

Let $L\subset \Rn$ be a subset containing the origin $o$ in its interior. For such an $L$, if the line segment $[0, x]\subset L$ for all $x\in L$, then $L$ is called a 
star-shaped set. The radial function of  a star-shaped set $L$ is denoted by $\rho_L:\R^n\setminus \{o\} \to \R$ and defined by $\rho_L(y)=\sup\{\tau:\tau y\in L\}.$ A compact star-shaped set with continuous radial function will be called a star body (with respect to the origin), and the set of all star bodies (with respect to the origin) is denoted by $\mathcal{S}_0^n$. In general, a set $E\subset \Rn$ is said to be star body with respect to $x$ if $E-x\in \mathcal{S}_0^n.$  Clearly, every $K\in\conbod$ is a star body with respect to every $x\in\text{int}(K)$. The \textit{roof function} for convex bodies plays important roles in the characterization of equality in our main results. For a convex body $K$ containing the origin, its roof function can be formulated by 
\begin{equation}
\aff K(x) = 
\begin{cases} 1 & \ \ \  x = o, \\ 
 0 & \ \ \  x\not\in K,  \\
\left(1-\frac{1}{\rho_K(x)}\right) &  \ \ \  x\in K\setminus\{o\}.
\end{cases}
\end{equation} We often use its equivalent form in  polar coordinates: for $\theta\in \s$ and $r\in [0,\rho_K(\theta)],$ 

\begin{equation}
\label{eq:lk}
\aff K(r\theta)=\left(1-\frac{r}{\rho_K(\theta)}\right),
\end{equation} and $\aff K(r\theta)=0
$ if $r>\rho_K(\theta).$ When $K\in\conbod_0,$ it is often more convenient to write $\aff K(x)$ as follows: 
$\aff K(x)=1-\|x\|_K$ for $x\in K$ and $0$ otherwise, where $\|\cdot\|_K$ stands for the \textit{Minkowski functional} of $K$, defined as $\|x\|_K=\rho_K(x)^{-1}.$

A function $\psi:\R^n\to\R$ is said to be a concave function if $$\psi((1-\tau)x+\tau y)  \geq (1-\tau)\psi(x)  +\tau \psi(y) $$
holds for every $x,y\in\text{supp}(\psi)$, the support of $\psi$,  and $\tau\in[0,1].$ A non-negative function $\psi$ is $s$-concave, $s>0$, if $\psi^s$ is a concave function, and is log-concave if $\log \psi$ is concave. The log-concavity can be obtained by $s$-concavity by letting $s\rightarrow 0^+$. A direct application of Jensen's inequality shows that for any $s > 0$, an $s$-concave function is also log-concave. 

 We shall need the following result which states how a log-concave function can be used to construct a convex body in $\conbod_0$. 

\begin{proposition}[Theorem 5 in \cite{Ball88} and Corollary 4.2 in \cite{GZ98}]
\label{p:radial_ball}
    Let $f$ be a log-concave function on $\Rn$. Then, for every $p> 0,$ the function on $\s$ given by 
    $$\theta\mapsto\left(p\int_0^\infty f(r\theta)r^{p-1}dr\right)^{1/p}$$
    defines the radial function of a convex body containing the origin in its interior.
\end{proposition}

Let $K$ be a convex body in $\Rn.$ For $p>-1, p\neq 0$, 
the function   
\begin{equation}
    \rho_{R_p K}(\theta)=\left(\frac{1}{\vol_n(K)}\int_K \rho_{K-x}(\theta)^p dx\right)^\frac{1}{p}, \ \ \ \  \theta\in\s
    \label{pth} 
\end{equation} is well defined on $\s$ and defines the radial function of a star body $R_pK$, i.e., the  \textit{pth radial mean body} of $K$ introduced by  Gardner and Zhang \cite{GZ98}. By appealing to continuity in $p$, $R_0 K$ and $R_\infty K$ can be defined as well, and in fact, one has $R_\infty K=DK$ \cite{GZ98}. Note that $R_p K$ tends to $\{o\}$  as  $p\to -1$, and hence, to obtain an interesting limit at $-1$ another family of star bodies depending on $K\in\conbod$ is needed. These new star bodies are called the \textit{pth spectral mean bodies} of $K$ \cite{GZ98} and are defined as follows: the $0$th spectral mean body is $e \cdot R_0 K$ and the $p$th spectral mean body, for $p\in (-1,0)\cup (0,\infty)$, is $(p+1)^\frac{1}{p}R_p K.$ This renormalization naturally brings the polar  projection body into the new family, as one has
 $$(p+1)^\frac{1}{p}R_p K \to \vol_n(K) \Pi^\circ K\ \ \ \mathrm{as} \ \ p\to (-1)^+.$$ 
 
By using Berwald's inequality \cite{Berlem,Bor73}, Gardner and Zhang \cite[Theorem 5.5]{GZ98} obtained that, for $-1 < p \leq q < \infty,$ 
 \begin{equation}
    \label{e:set_inclusion}
    D K \subseteq \binom{n+q}{q} R_{q} K \subseteq \binom{n+p}{p} R_{p} K \subseteq n \vol_n(K) \Pi^{\circ} K,
\end{equation} with equality in each inclusion in \eqref{e:set_inclusion} if, and only if, $K$ is a $n$-dimensional simplex. Note that for any $p \ge 0$, $R_p K$ is an origin symmetric convex body (as can be easily checked by applying Proposition~\ref{p:radial_ball} to the covariogram function), however the convexity of $R_p K$  for $p\in (-1,0)$ is still unknown. Extension of the radial mean bodies themselves in different settings can be found in \cite{HLPRY23,HL22,LP23}.

Many of our results require the measure $\mu$ to have some concavity. To this end, let $F:(0,\mu(\R^n))\to (-\infty,\infty)$ be a continuous, invertible and strictly monotonic function. We say that a Borel measure $\mu$ is $F$-concave on a class $\mathcal{C}$ of compact Borel subsets of $\R^n$ if  \begin{equation}
  \label{eq:concave}
  \mu(\tau A +(1-\tau)B) \geq F^{-1}\left(\tau F(\mu(A)) +(1-\tau)F(\mu(B))\right) \end{equation} for any $A,B \in \mathcal{C}$  and  $\tau \in [0,1]$. When $\mu$ satisfies \eqref{eq:concave} for   
  $F(t)=t^s$, $ s \in\R\setminus\{0\}$,  $\mu$ is said to be a $s$-concave measure, while $\mu$ is a log-concave measure if $\mu$ satisfies \eqref{eq:concave} for $F(t)=\log t$. In particular, the Lebesgue measure $\lambda $ on $\Rn$ is a $1/n$-concave measure  on the class of compact subsets of $\R^n$, due to the Brunn-Minkowski inequality. In fact, Borell's classification of concave measures \cite{Bor75} states that a Radon measure is log-concave on Borel subsets of $\R^n$ if, and only if,  $\mu$ has a density $\phi(x)$ that is log-concave, i.e., $\phi(x)=e^{-\psi(x)},$ where $\psi:\R^n\to\R$ is convex. Similarly, a Radon measure is $s$-concave on Borel subsets of $\R^n$, $s>0,$ if, and only if, $\mu$ has a density $\phi$ that is zero almost everywhere if $s>1/n$, is constant if $s=1/n$, or is $s/(1-ns)$-concave if $s\in (0,\frac{1}{n}$).  

The following result asserts that the $\mu$-covariogram inherits the concavity of the measure $\mu$.
\begin{proposition}[Concavity of the covariogram, \cite{LRZ22}]
\label{p:covario_concave}
Consider a class of convex bodies $\mathcal{C}\subseteq\conbod$ with the property that $K\in \mathcal{C} \rightarrow K\cap(K+x)\in\mathcal{C}$ for every $x\in DK$. Let $\mu$ be a Borel measure finite on every $K\in\mathcal{C}.$ Suppose $F$ is a continuous and invertible function such that $\mu$ is $F$-concave on $\mathcal{C}$. Then, for $K\in\mathcal{C},$ $g_{\mu,K}$ is also $F$-concave, in the sense that, if $F$ is increasing, then $F\circ g_{\mu,K}$ is concave, and if $F$ is decreasing, then $F\circ g_{\mu,K}$ is convex.
\end{proposition}

We shall need the following result regarding some properties of concave functions.
\begin{lemma}[Lemma 2.4, \cite{LRZ22}]
	\label{l:concave}
	Let $f$ be a concave function that is supported on  $L\in\conbod_0$ such that
	$$
h(\theta) := \frac{df(r\theta)}{dr}\bigg|_{r=0^+} < 0, \quad  \text{for all }  \theta\in \s,
$$ 
and $f(o) > 0.$ Define $z(\theta)=-\left(h(\theta)\right)^{-1}f(o).$ Then,
\begin{equation}\label{eq:concave_function} -\infty < f(r\theta)\leq f(o)\left[1-(z(\theta))^{-1}r\right]\end{equation}
whenever $\theta\in\s$ and $r\in[0,\rho_L(\theta)]$. In particular, if $f$ is non-negative, then we have
$$0 \leq f(r\theta)\leq f(o)\left[1-(z(\theta))^{-1}r\right] \quad  \mbox{and }  \ \  \ \rho_L(\theta)\leq z(\theta).
$$
In this case, $ f(r\theta)=f(o)\left[1-(z(\theta))^{-1}r\right]$ for $r\in[0,\rho_L(\theta)]$ if, and only if, $\rho_L(\theta)=z(\theta).$
\end{lemma}

Recall that $\mathcal{M}(K)$ is the set of Borel measures $\mu$ on $\Rn$ such that $\mu$ has locally integrable density $\phi \colon \R^n \to \R^+$ containing $K$ in its support and $\partial K$ in its Lebesgue set. Let $K\in\conbod$ and $\mu\in  \mathcal{M}(K)$ be $F$-concave, where $F$ is a non-negative, differentiable, strictly increasing function. It follows from \eqref{covario}, \eqref{e:deriv_g_mu_covario} and Proposition \ref{p:covario_concave} that $\supp(F \circ g_{\mu, K}) = DK$,  $F \circ g_{\mu, K}$ is concave and  
\begin{align*}
    \frac{d}{dr} (F \circ g_{\mu, K})(r\theta)\bigg|_{r=0^+}  &= F'(g_{\mu, K}(o)) \frac{d}{dr} g_{\mu, K}(r\theta)\bigg|_{r=0^+}  
     \\
     &= -F'(\mu(K)) h_{\Pi_\mu K}(\theta)\ \ \ \mathrm{for\ \ any\ } \theta \in \s.
\end{align*} Moreover, applying Lemma \ref{l:concave} to $f(r\theta)=(F \circ g_{\mu, K})(r\theta)$ and $L=DK$, one has 
\begin{equation}
    \rho_{DK}(\theta) \le \frac{F(\mu(K))}{F'(\mu(K)} h_{\Pi_\mu K}(\theta)^{-1} = \frac{F(\mu(K))}{F'(\mu(K)} \rho_{\Pi_\mu^\circ K}(\theta). 
\end{equation} This yields
\begin{equation}\label{e:mu_set_incl}
DK\subseteq \frac{F(\mu(K))}{F^\prime(\mu(K))} \Pi^\circ_{\mu}K.
\end{equation} 

\section{A Higher-Order Weighted Projection Body Operator}
\label{sec:higher_dim}

In this section, we will introduce the $m$th-order weighted projection body operator and prove some of its properties.  
We first establish the concavity of $g_{\mu,\hid}(K,\cdot)$ defined in Definition~\ref{eq_higher_co}, namely,  
\[
g_{\mu,\hid}(K, \bar x) = g_{\mu,\hid}(K,(x_1,\dots,x_{\hid})) = \int_{K} \left(\prod_{i=1}^{\hid} \chi_{K}(y-x_i)\right)\,d\mu(y).
\]

\begin{lemma}
\label{l:higher_covario} Let $\hid \in \N$, $K \in \conbod$, and $F$ be a positive, continuous and invertible function. Let $\mu$ be a Borel measure on $\R^n$ such that $\mu$ is $F$-concave on a class of convex bodies $\mathcal{C}\subseteq\conbod$ with the property that $K\in \mathcal{C}$ implies $K\cap_{i=1}^\hid (K+x_i)\in\mathcal{C}$ for every $\bar x=(x_1,\dots,x_\hid)\in D^\hid (K)$. Then, for $K\in\mathcal{C},$ $g_{\mu,\hid}(K,\cdot)$ is also $F$-concave, in the sense that, if $F$ is increasing, then $F\circ g_{\mu,\hid}(K,\cdot)$ is concave, and if $F$ is decreasing, then $F\circ g_{\mu,\hid}(K,\cdot)$ is convex.
\end{lemma}

\begin{proof} For any $t \in [0,1]$,
$\overline{x}= (x_1,\dots,x_{\hid})\in D^\hid(K)$ and $\overline{y} = (y_1,\dots,y_{\hid})\in D^\hid(K)$, let 
\begin{align}\label{def-K^t-1}
K^t(\overline{x}, \overline{y}) = K \cap \left[ \bigcap_{i=1}^{\hid} ((1-t)x_i+ty_i + K) \right].
\end{align} The desired result will follow once the following is verified:  
\begin{equation}\label{eq:kt_supset}
K^t(\overline{x}, \overline{y}) \supseteq (1-t)  \left[K \cap \left(\bigcap_{i=1}^{\hid} (x_i +K)\right) \right] + t \left[K \cap \left(\bigcap_{i=1}^{\hid} (y_i +K)\right) \right],
\end{equation} 
 Indeed, suppose $F$ is increasing. This set inclusion, together with \eqref{eq:concave},  then yields
\begin{align*}&F(g_{\mu, m}(K, (1 - t)\bar x + t\bar y))= F(\mu(K^t(\overline{x}, \overline{y}))) \\
&\ge  F\left(\mu\left((1-t)  \left[K \cap \left(\bigcap_{i=1}^{\hid} (x_i +K)\right) \right] + t \left[K \cap \left(\bigcap_{i=1}^{\hid} (y_i +K)\right)\right]\right)\right) \\
&\ge (1 - t) F\left(\mu\left(K \cap \left(\bigcap_{i=1}^{\hid} (x_i +K)\right)\right)\right) + t F\left(\mu\left(K \cap \left(\bigcap_{i=1}^{\hid} (y_i +K)\right)\right)\right) \\
&= (1 - t) F(g_{\mu, m}(K, \bar x)) + t F(g_{\mu, m}(K, \bar y)),
\end{align*}
where the third line uses the $F$-concavity of $\mu$. This shows that $F\circ g_{\mu,\hid}(K,\cdot)$ is concave. Similar computations show that if $F$ is decreasing, then $F\circ g_{\mu,\hid}(K,\cdot)$ is convex.

Now we show that \eqref{eq:kt_supset} holds. To this end, let $$\Bar{z} \in (1-t)  \left[K \cap \left(\bigcap_{i=1}^{\hid} (x_i +K)\right) \right] + t \left[K \cap \left(\bigcap_{i=1}^{\hid} (y_i +K)\right) \right].$$ Then $\Bar{z} = (1-t) z+tz'$, with 
\[
z \in K \cap \left(\bigcap_{i=1}^{\hid} (x_i +K)\right) \text{ and } z' \in K \cap \left(\bigcap_{i=1}^{\hid} (y_i +K)\right).
\]
By the convexity of $K$, we see that $\Bar{z} \in K$. For each $i =1,\dots,\hid$, there exist $\tilde{z}_i,\tilde{z}_i' \in K$ such that $z=x_i+\tilde{z}_i$ and $z'=y_i+\tilde{z}_i'$, which means that $$\Bar{z} = (1-t)x_i +t y_i + ((1-t)\tilde{z}_i+t\tilde{z}_i')\in (1-t)x_i +t y_i +K$$ holds for every $i = 1,\dots,\hid$. It then follows that $\Bar{z} \in K^t(\bar{x},\Bar{y}),$ as required. 
\end{proof}

We now introduce our main tool in proving Theorem~\ref{t:variationalformula}, Aleksandrov's variational formula for arbitrary measures. For a continuous function $h: \s \to (0, \infty)$, the Wulff shape or Alexandrov body of $h$ is defined as 
\begin{align} \label{wullf-1-1} [h] = \bigcap_{u \in \s} \{x \in \mathbb R^n: \langle x, u\rangle \le h(u)\}. \end{align}

\begin{lemma}[Aleksandrov's variational formula for arbitrary measures, Lemma 2.7 in \cite{KL23}]
	\label{l:second}
Let $K$ be a convex body, let $\mu\in\mathcal{M}(K)$, and let $f$ be a continuous function on $\s$. Then
	$$\lim_{t\rightarrow 0}\frac{\mu([h_K+tf])-\mu(K)}{t}=\int_{\s}f(\theta)\, dS^{\mu}_K(\theta).$$
	\end{lemma}	
 
% \textcolor{red}{
% Let $K$ be general. Set $x_0=\frac{1}{\vol_n(K)}\int_K x dx$. Then, notice for $t$ small enough:
% $$[h_K + tf]=[h_{K-x_0}+tf]+x_0$$
% Let $\phi$ be density of $\mu$. Let $\tilde \phi = \phi(x-x_0)$. Then, $\mu([h_K + tf])=\tilde \mu([h_{K-x_0}+tf])$ and $\mu(K)=\tilde \mu(K+x_0).$ Then:
% $$\lim_{t\rightarrow 0}\frac{\tilde\mu([h_K+tf])-\tilde\mu(K)}{t}=\int_{\partial K + x_0}\phi(x-x_0)\mathcal{H}^{n-1}(x).$$
% }

This extends the result of \cite{GHX19}, which proves the same formula under the assumption that $K$ has the origin in its interior and that $\mu$ has continuous density.  We now prove Theorem \ref{t:variationalformula}, which we restate here for convenience. 

\vskip 2mm \noindent {\bf Theorem \ref{t:variationalformula}.} {\em Let $K\in\conbod$, $\hid \in \N$, and $\phi$ be the density of a Borel measure $\mu\in\mathcal{M}(K)$. For every direction $\bar\theta = (\theta_1,\dots,\theta_{\hid}) \in \S$, one has 
\[
\frac{d}{dr} \left[g_{\mu,\hid}(K,r\bar\theta) \right]\bigg|_{r=0^+}
= -\int_{\partial K} \max_{1 \leq i \leq \hid} \langle \theta_i, n_K(y) \rangle_{-} \phi(y)\,dy.\] 
 }

\begin{proof} It can easily be checked that  $h_{K + r\theta}(u) = h_K(u) + r\langle u, \theta\rangle$ for any $r\geq 0$ and for any $\theta \in \mathbb R^n$. Moreover, any convex body $L$ is the Wulff shape of its support function $h_L$, i.e., $L = \bigcap_{u \in \s} \{x: \langle u, x\rangle \le h_L(u)\}$. For notational convenience,   let $\theta_0 = 0$ and $K_r = K \cap (K + r \theta_1) \cap \cdots \cap (K + r \theta_i)$. Then 
 \begin{align*}
K_r &= \bigcap_{i = 0}^\hid \bigcap_{u \in \s} \{x: \langle u, x\rangle \le h_{K + r\theta_i }(u)\} \\
&= \bigcap_{u \in \s} \bigcap_{i = 0}^\hid \{x: \langle u, x\rangle \le h_{K + r\theta_i }(u)\} \\
&= \bigcap_{u \in \s} \{x:  \langle u, x\rangle \le \min_{0\le i \le \hid} (h_{K}(u) + r\langle \theta_i, u\rangle)\} \\
&= \bigcap_{u \in \s} \{x:  \langle u, x\rangle \le h_K(u) + r\min_{0\le i \le \hid} \langle u, \theta_i\rangle\}\\ &=[h_K(u) + r \min_{0\le i \le \hid} \langle u, \theta_i\rangle],
\end{align*} where the last equality follows from \eqref{wullf-1-1}. It follows from \eqref{eq_vol_ell_co-mu} that $$g_{\mu,\hid}(K,r\bar\theta) =\mu(K_r).$$ Applying Lemma~\ref{l:second} to 
$$f(u) = \min_{0 \le i \le \hid} \langle u, \theta_i\rangle = \min_{1 \le i \le \hid} (-\langle u, \theta_i\rangle_-)=- \max_{1 \le i \le \hid} \langle u, \theta_i\rangle_-,$$ one gets, with the help of Definition \ref{def:surface_mu},
\begin{align*}
    \frac{d}{dr} \left[g_{\mu,\hid}(K,r\bar\theta) \right]\bigg|_{r=0^+} &= -\int_{\s} \max_{1 \le i \le \hid} \langle u, \theta_i\rangle_-\,dS^{\mu}_K(u) 
    \\
    &=-\int_{\partial K} \max_{1 \le i \le \hid} \langle n_K(x), \theta_i\rangle_-\,\phi(x)\,d\mathcal H^{n - 1}(x).
\end{align*} This completes the proof. 
\end{proof}

Henceforth we will suppose that $\phi$, the density of $\mu$, is strictly positive on $K$. Thus, for each $\bar{\theta}\in \mathbb{S}^{mn-1},$ \begin{align}\label{positive-derivative-1} \frac{d}{dr} \left[g_{\mu,\hid}(K,r\bar\theta) \right]\bigg|_{r=0^+}
= -\int_{\partial K} \max_{1 \leq i \leq \hid} \langle \theta_i, n_K(y) \rangle_{-} \phi(y)\,dy<0.\end{align}
Following the argument leading to \eqref{e:mu_set_incl}, by using Lemma~\ref{l:concave} and Lemma~\ref{l:higher_covario}, it can be checked that \begin{equation}\label{e:mu_set_incl_high}
D^\hid(K)\subseteq \frac{F(\mu(K))}{F^\prime(\mu(K))} \Pi_{\mu}^{\circ,\hid} K,
\end{equation} where $F$ is a non-negative, differentiable and strictly increasing function, $K\in\conbod$ is a convex body, and $\mu\in \mathcal{M}(K)$ is an $F$-concave Borel measure.

 For an invertible linear map $T: \Rn \to \Rn$, we define $\overline{T}: \R^{mn}\to \R^{mn}$ by $\overline{T}(\bar x)= (T(x_1),\dots,T(x_n))$ where $\bar x=(x_1, \dots, x_m)\in \Rnhi.$ Note that $\overline{T}$ is an invertible linear map on $\R^{mn}.$  Denote by $|\det T|$ the absolute value of the determinant of $T$. For a Borel measure $\mu,$ define  $$\mu^T = |\det T|^{-1} (T^{-1})_* \mu$$ where $(T^{-1})_* \mu$ is the pushforward of $\mu$ under $T^{-1}$. One verifies that $\mu^T$ is absolutely continuous with respect to the Lebesgue measure $\lambda$ and satisfies $d\mu^T(x) = \phi(Tx)\,d\lambda(x).$

 We now determine the behavior of $\P_\mu  K$ under linear transformations. 
 
\begin{proposition}
\label{prop:linear:transformations}
    Let $T$ be an invertible linear map on $\Rn$, $K\in\conbod$, and $\mu\in \mathcal{M}(K)$.   For $\hid \in \mathbb{N}$, one has
    \begin{equation}
        \P_\mu(TK) = |\det T| \cdot \overline{T^{-t}}\Pi^\hid_{\mu^T}K,
    \end{equation}
    where $|\det T|$ is the absolute value of the determinant of $T$. \end{proposition}
\begin{proof}

    Apply Theorem~\ref{t:variationalformula} to $TK$ to obtain that, for $\bar{\theta}=(\theta_1, \dots, \theta_m)\in \mathbb{S}^{mn-1}$, 
    \begin{align*}h_{\P_\mu  TK}(\bar{\theta})&=-\lim_{t\to 0^+}\frac{\mu\left((TK)\cap_{i=1}^\hid (TK+t\theta_i)\right)-\mu(TK)}{t}
    \\
    &=-\lim_{t\to 0^+}\frac{\mu\left(T\left(K\cap_{i=1}^\hid (K+tT^{-1}\theta_i)\right)\right)-\mu(TK)}{t}
    \\
    &=-|\det T|\lim_{t\to 0^+}\frac{\mu^T\left(K\cap_{i=1}^\hid (K+tT^{-1}\theta_i)\right)-\mu^T(K)}{t}
    \\
    &=|\det T|\ h_{\P_{\mu^T}  K}(\overline{T^{-1}}(\bar{\theta}))=h_{|\det T|\P_\mu  K}(\overline{T^{-1}}(\bar{\theta}))
    \\
    &=h_{\overline{T^{-t}}|\det T|\P_{\mu^T}  K}(\bar{\theta}),
    \end{align*}
    and the claim follows.
\end{proof}

Recall that a function is said to be radially non-decreasing if for every $t\in [0,1]$ and $x\in\Rn$, one has $\varphi(tx) \leq \varphi(x)$. We shall need the following chord integral inequality. We recall that the support of a function is the set
$$\text{supp}(f)=\{x\in\R^n:f(x)\neq 0\}.$$

\begin{lemma}[Lemma 3.2 from \cite{LRZ22}]\label{t:chak}
Let $\mu$ be a Borel measure on $\Rn$ with radially non-decreasing, locally integrable density $\phi$ (with respect to the Lebesgue measure $\lambda$), and let $f:\R^n\to\R^{+}$ be a compactly supported, concave function such that $o\in \text{int}(\supp(f))$ and $f(o)=\max f(x).$ Set $$\Omega_f:=\left\{\theta\in\s: \frac{df(r \theta)}{dr}\bigg|_{r=0^+}=0\right\}.$$
If $q \colon \R^{+} \to \R$ is an increasing function, then
\begin{equation*}
\begin{split}
\int_{\supp(f)} q(f(x))\, d\mu(x) &\leq \beta \int_{\s\setminus{\Omega_f}}\int_0^{z(\theta)}\phi(r\theta)r^{n-1}\,dr\,d\theta 
\\
&\quad+\int_{\Omega_f}\int_0^{\rho_{\supp(f)}(\theta)}q(f(o))\phi(r\theta)r^{n-1}\,dr\,d\theta,
\end{split}
\end{equation*}
where 
$$
z(\theta)=-\left(\frac{df(r\theta)}{dr}\bigg|_{r=0^+}\right)^{-1}f(o) \quad \text{and } \quad \beta =n\int_0^1 q(f(o)t) (1-t)^{n-1} dt.$$
Equality occurs if, and only if: 
\begin{enumerate}
    \item for $\theta\in\s\setminus\Omega_f$, one has $z(\theta)=\rho_{\supp(f)}(\theta),$ $f(r\theta)$ is an affine function for $r\in [0,\rho_{\supp(f)}(\theta)]$, and, for every $r>0,$ $\phi(r\theta)$ is independent of $r$; 
    \item for almost every $\theta\in\Omega_f,$ one has $f(r\theta)=f(o)$ for every $r\in [0,\rho_{\supp(f)}(\theta)].$
\end{enumerate}
\end{lemma}

 We are now ready to prove the analogue of Zhang's projection inequality for $\P_\mu  K$. We remark that even when $\hid=1,$ the results presented here are strictly stronger than those of \cite{LRZ22}; in that work, there was a Lipschitz assumption on the density of the measure $\mu$.

 \begin{theorem}\label{t:ZhangInequality}
Fix $\hid \in \N$ and $K\in\conbod.$ Let $\nu_1,\dots,\nu_{\hid}$ be Borel measures on $\R^n$, each having radially non-decreasing density, and set $\nu = \nu_1\otimes \cdots \otimes \nu_{\hid}$ to be the associated product measure on $\Rnhi$. Let $F \colon \R^+ \to \R^+$ be a strictly increasing and differentiable function. Let a Borel measure $\mu$ on $\R^n$ be $F$-concave on a class of convex bodies $\mathcal{C}\subseteq\conbod$ with the property that $K\in \mathcal{C}$ yields $\mu\in\mathcal{M}(K)$ and $K\cap_{i=1}^\hid (K+x_i)\in\mathcal{C}$ for every $(x_1,\dots,x_\hid)\in D^\hid (K)$.
Then, 
\begin{equation}\label{eq:zhang_ineq}
\nu\left(\frac{F(\mu(K))}{F'(\mu(K))} \PP_{\mu}K\right) \geq \frac{\int_K \prod_{i=1}^{\hid}\nu_i(y-K)\, d\mu(y)}{n\hid \int_0^1 F^{-1}[F(\mu(K))t](1-t)^{n\hid-1}dt}.
\end{equation}
Equality occurs if, and only if, the following are true:
\begin{enumerate}
    \item If $\varphi$ is the density of $\nu,$ then, for each $\bar\theta\in\S$, $\varphi(r\bar\theta)$ is independent of $r,$ and
    \item for each $\bar\theta\in\S,$ $F\circ g_{\mu,\hid}(K,r\bar\theta)$ is an affine function in the variable $r$ for $r\in[0,\rho_{D^\hid(K)}(\bar\theta)]$, which is also equivalent to $$D^\hid(K)= \frac{F(\mu(K))}{F^\prime(\mu(K))} \PP_{\mu} K.$$
\end{enumerate}
\end{theorem}

\begin{proof} Our goal is to estimate 
\[
I := \int_{\R^{n\hid}} g_{\mu,\hid}(K, \overline{x})\, d \nu(\overline{x}),
\]  where $g_{\mu,\hid}(K, \bar x)$ is given in Definition~\ref{eq_higher_co}. Note that, from Lemma~\ref{l:higher_covario}, $F\circ g_{\mu,\hid}(K, \bar x)$ is concave on its support. It follows from Fubini's theorem that 
\begin{align*}
I &= \int_{\R^n} \cdots \int_{\R^n} \mu \left[K \cap \left(\bigcap_{i=1}^{\hid} (x_i +K)\right) \right] d\nu_{\hid}(x_{\hid}) \cdots d\nu_1(x_1)\\
&= \int_{\R^n} \cdots \int_{\R^n} \left( \int_K \prod_{i=1}^{\hid} \chi_{y-K}(x_i) \,d\mu(y) \right)\, d\nu_{\hid}(x_{\hid}) \cdots d\nu_1(x_1)\\
&= \int_K \prod_{i=1}^{\hid}\nu_i(y-K)\, d\mu(y).
\end{align*}
We now apply Lemma~\ref{t:chak} with $
f = F \circ g_{\mu,\hid}(K, \cdot)$ and $q = F^{-1}, $
together with \eqref{positive-derivative-1} and Theorem~\ref{t:variationalformula}. Note that $\Omega_f=\emptyset$, $f(o) = F(g_{\mu, \hid}(K, o)) = F(\mu(K))$, and  $z(\bar\theta) =\frac{F(\mu(K))}{F'(\mu(K))} \rho_{\PP_\mu K}(\bar\theta).$ Letting $\varphi$ denote the density of $\nu$, this yields
\begin{align*}
I &= \int_{D^\hid(K)} g_{\mu,\hid}(K, (\overline{x}))\, d\nu(\overline{x})\\ & = \int_{D^\hid(K)} F^{-1} [F \circ g_{\mu,\hid}(K, \overline x)] d\nu(\overline{x}) \\
&\leq \left(\int_{\S}\int_0^{\frac{F(\mu(K))}{F^\prime(\mu(K))}\rho_{\PP_\mu K}(\bar\theta)} \varphi(r\bar\theta)r^{n\hid-1}drd\bar\theta\right) 
\\
&\quad\quad\times \left( n\hid  \int_0^1 F^{-1}[F(\mu(K))t](1-t)^{n\hid-1}dt\right)
\\
&= n\hid \nu\left(\frac{F(\mu(K))}{F'(\mu(K))} \PP_\mu K \right) \int_0^1 F^{-1}[F(\mu(K))t](1-t)^{n\hid-1}dt,
\end{align*}
as desired. The equality conditions are inherited from Lemma~\ref{t:chak} as well.
\end{proof}

Notice that, as $F:\R^+\to\R^+$ is strictly increasing, one has $$
\int_0^1 F^{-1}[F(\mu(K))t](1-t)^{n\hid-1}dt\leq  \mu(K) \int_0^1 (1-t)^{n\hid-1}dt =\frac{\mu(K)}{mn}.$$ Thus, \eqref{eq:zhang_ineq} implies in particular that\begin{equation}\label{eq:zhang_ineq-simplified}
\nu\left(\frac{F(\mu(K))}{F'(\mu(K))} \PP_{\mu}K\right) \geq \int_K \frac{\prod_{i=1}^{\hid}\nu_i(y-K)}{\mu(K)}\, d\mu(y).
\end{equation}  In the case when $F(t) = t^s$, i.e., the measure in Theorem \ref{t:ZhangInequality} is assumed to be $s$-concave, then \eqref{eq:zhang_ineq}  becomes 
\begin{equation}\label{eq:zhang_ineq-s-concave}
\nu\left(\frac{\mu(K)}{s} \cdot \PP_{\mu}K\right) \geq \frac{\int_K \prod_{i=1}^{\hid}\nu_i(y-K)\, d\mu(y)}{n\hid \mu(K) \int_0^1 t^{1/s}(1-t)^{n\hid-1}dt}.
\end{equation} After a rearrangement, one gets 
\begin{equation}\label{eq:zhang_s_conc}
\frac{\mu(K)\nu\left(\frac{1}{s}\mu(K)\PP_{\mu} K\right)}{\int_K \prod_{i=1}^{\hid}\nu_i(y-K)\, d\mu(y)} \geq \binom{n\hid + \frac{1}{s}}{n\hid}.
\end{equation}

Moreover, in this case the equality conditions are quite simple to state: $K$ must be a $n$-dimensional simplex and the density $\phi$ of $\mu$ must be constant on $K$ -- that is, equality holds only in the classical case. These results are summarized in the following corollary. 

\begin{corollary}\label{cor:ZhangInequality}
Fix $\hid \in \N$ and $K\in\conbod.$ Let $\nu_1,\dots,\nu_{\hid}$ be Borel measures on $\R^n$, each having radially non-decreasing density, and set $\nu = \nu_1\otimes \cdots \otimes \nu_{\hid}$ to be the associated product measure on $\Rnhi$. Let $s > 0$ and let $\mu$ be a locally finite, regular, $s$-concave  measure on $\R^n$. Then, inequality \eqref{eq:zhang_s_conc} holds with equality if, and only if, the following are true:
\begin{enumerate}
    \item[(i)] If $\varphi$ is the density of $\nu,$ then, for each $\bar\theta\in\S$, $\varphi(r\bar\theta)$ is independent of $r,$
    \item[(ii)] $K$ is a $n$-dimensional simplex,
    \item[(iii)] the density of $\mu$ is constant on $K$,
    \item[(iv)] $s = 1/n$.
\end{enumerate}
\end{corollary}

In the following, we will provide a detailed proof for the equality characterization in Corollary \ref{cor:ZhangInequality}. To fulfill this goal,  we need to introduce a bit more background. First, recall that the following are equivalent (see e.g.,  \cite[Section 6]{EGK64}, or \cite{Choquet,RS57}):
	\begin{enumerate}
        \item[(i)] $K$ is a $n$-dimensional simplex.
	    \item[(ii)] \label{item:hom} For any $x \in \R^n$ such that $(K + x) \cap K \neq \emptyset$, $K\cap (K+ x)$ is homothetic to $K$, namely, there exist a constant $a>0$ and a vector $x_0\in \Rn,$ such that $K\cap (K+ x)=aK+x_0=\{ax+x_0: x\in K\}$.
     \end{enumerate}

Next, we recall a result of Milman and Rotem \cite[Corollary 2.16]{MR14}:

\begin{lemma}\label{l:milman_rotem} Let $\mu$ be a locally finite, regular, $s$-concave measure on $\R^n$ with density $\phi$. Suppose that $t \in (0, 1)$ and $A, B \subset \mathbb R^n$ are Borel sets of positive measure satisfying 
$$\mu(t   A + (1 - t) B)^s = t \mu(A)^s + (1 - t) \mu(B)^s.$$
Then up to $\mu$-null sets, there exist $c, m > 0$, $b \in \R^n$ such that $B = mA + b$ and $\phi(mx + b) = c \cdot \phi(x)$ for all $x \in A$. 
\end{lemma}

  We can now prove the following proposition. 

\begin{proposition}
	\label{p:simp}
	Fix $K\in\conbod$, $\hid\in\N$, and $s>0$. Let $\mu$ be a locally finite,  regular, and $s$-concave Borel measure on compact subsets of the support of its density $\phi$, which contains $K$. Then, for every $\theta\in\s$, $g_{\mu,\hid}(K;r\theta)^s$ is an affine function in $r$ for $r\in[0,\rho_{D^\hid (K)}(\theta)]$ if, and only if, $K$ is an $n$-dimensional simplex, $\phi$ is constant on $K$, and $s = 1/n$.
 \end{proposition}

 \begin{proof} Note that $D^\hid(K)$ is the support of $g_{\mu,\hid}(K,\cdot)$. Let $\bar x \in \mathrm{int}(D^\hid(K))$ and  $t \in (0, 1)$. The fact that $g_{\mu, m}(K, \cdot)$ is affine on the segment $[o, \bar x]$ precisely means that for all $t \in (0, 1)$,
\begin{equation}\label{eq:zhang_eq}
\mu(K^t(o, \bar x))^s = (1-t) \mu(K)^s + t \mu(K^1(o, \bar x))^s,
\end{equation} 
where $K^t(o, \bar y) = K \bigcap_{i = 1}^m (K + t y_i)$ as in \eqref{def-K^t-1}. Examining the proof of Lemma \ref{l:higher_covario}, we see that
$K^t (o, \bar x) \supseteq (1 - t) K + t K^1(o, \bar x)$,  and equality can hold in \eqref{eq:zhang_eq} only if $K^t(o,  \bar x) = (1 -t) K +t K^1(0, \bar x)$.  In particular, we have
\begin{equation}\label{eq:zhang_eq_2}
\mu( (1 - t) K + t K^1(0, \bar x))^s =(1-t) \mu(K)^s + t\mu(K^1(0, \bar x))^s.
\end{equation}

By Lemma \ref{l:milman_rotem} and the fact that both $K$ and $K^1(0, \bar x)$ are convex,  it follows that  
$K$ is homothetic to $K^1(0, \bar x)=K \cap (K + x_1) \cap \cdots \cap (K + x_m)$ for any $\bar x\in \mathrm{int} (D^\hid (K))$. Setting $x_2 = \cdots = x_m = o,$ we have in particular that $K \cap (K + x_1)$ is homothetic to $K$ for all $x_1\in \mathrm{int}(DK)$. This implies that $K$ is a $n$-dimensional simplex.

It remains to show that the density of $\mu$ is constant on $K$. To this end,  we use the second conclusion of Lemma \ref{l:milman_rotem} (with $A=K$ and $B=K\cap (K+x)$ for arbitrary $x \in \mathrm{int}(DK)$): there exists $c(x) > 0$ such that for each $y \in K$, $\phi(A_x y) = c(x) \phi(y)$, where $A_x$ is the (unique) affine transformation which maps $K$ onto $K \cap (K + x)$. But note that $A_x$ is a continuous map from the compact, convex set $K$ to itself, so it has a fixed point $y$ due to Brouwer's fixed point theorem. For such $y$, we have $\phi(y) = \phi(A_x y) = c(x) \phi(y)$, implying $c(x) = 1$.

For any $x \in K$ there exists a vertex $v$ of $K$ such that $v$ belongs to any face containing $x$. Noting that $K - v + x$ is an $n$-simplex whose vertex corresponding to $v$ is at $x$, one sees that $K_x = K \cap (K - v + x)$ is an $n$-simplex homothetic to $K$, with the homothety sending $x$ to $v$; hence we have $\phi(x) = \phi(v)$. If $x$ is an interior point of $K$, any vertex of $K$ satisfies the above condition, which shows that $\phi(u) = \phi(v)$ for any two vertices $u, v \in K$, and hence that $\phi$ is constant on $K$. %Once one knows that $\phi(A_x y) = \phi(y)$ for any homothety $A_x$ mapping $K$ to $K \cap (K + x)$ and any $y \in K$, one verifies by tedious but elementary arguments (e.g., by starting with the faces of $K$, working by induction on the dimension) that any two points in $K$ can be mapped into each other by a chain of such $A_x$'s, which implies that $\phi$ is indeed constant on $K$. 

Thus we are back to the case of Lebesgue measure, which we know is affine when raised to the $1/n$ power on pairs of homothetic bodies, namely, $s=1/n$. 

Conversely, one can easily  verify that, if $K$ is a $n$-dimensional simplex and $\phi$ is constant on $K$,  then $g_{\mu, m}(K, \cdot)^{1/n}$ is affine on radial segments, as in the classical Zhang's projection inequality.
 \end{proof}

 \begin{proof} [Proof of Corollary \ref{cor:ZhangInequality}.] Suppose that equality holds in \eqref{eq:zhang_s_conc}. Condition (i) is precisely the same as condition (1) in the equality case of Theorem \ref{t:ZhangInequality}. Condition (2) of the equality case of Theorem \ref{t:ZhangInequality} states that $(g_{\mu, \hid}(K, r\bar\theta))^s$ is affine in $r$ for each $\bar\theta \in \S$, which, by Proposition~\ref{p:simp}, holds if and only if $K$ is a $n$-dimensional simplex, the density of $\mu$ is constant on $K$, and $s = 1/n$.
\end{proof}

 \section{Weighted $m$th-Order Radial Mean Bodies}
\label{sec:new_rad}
In this section, we will introduce the weighted $m$th-order radial mean bodies, and prove Theorem  \ref{t:set_s_con}. 

\begin{definition}\label{definition-m-radial mean body}
Let $\mu$ be a Borel measure on $\Rn$ with density $\phi$ containing $K\in \conbod$ in its support. For $\hid\in\N$ and $p >-1,$ the $\mu$-weighted \textit{$(\hid,p)$ radial mean body} $R^\hid_{p, \mu} K$ is defined as the star body in $\R^{n\hid}$ whose radial function for $\bar{\theta}\in\S$ is given by:
\begin{equation}
    \rho_{R_{p,\mu}^{\hid} K}(\bar\theta)^p=\frac{1}{\mu(K)} \int_K \left(\min _{i=1, \dots, \hid}\left\{\rho_{K-x}(-\theta_i)\right\}\right)^p d\mu(x)
\end{equation}
for $p\neq 0.$ The case $p=0$ follows from continuity of the $p$th average. We also let  $R^\hid_{\infty, \mu} K=\lim_{p \to \infty} R_{p,\mu}^{\hid} K$. 
\end{definition}

\begin{proposition}
    Fix $\hid,n\in \N$ and a Borel measure $\mu$ containing a convex body $K\in\conbod$ in its support. Then, $R^\hid_{\infty,\mu} K=D^\hid(K).$ 
\end{proposition}
\begin{proof}
Let $\bar y \in \Rnhi \backslash \{o\}$. Properties of $p$th averages yield, if we write $\bar y=(y_1,\dots,y_\hid),$
that $$\rho_{R^\hid_{\infty,\mu} K}(\bar y)=\max_{x\in K}\min_{i=1,\dots,\hid}\rho_{K-x}(-y_i)=\rho_{R^\hid_{\infty} K}(\bar y)=\rho_{D^\hid (K)}(\bar y).$$ That is,  $R^\hid_{\infty,\mu} K=D^\hid(K).$ 
\end{proof}

 Let $\psi:[0,\infty)\to[0,\infty)$ be an integrable function that is right continuous and differentiable at $0.$ Define the Mellin transform (see e.g., \cite{FLM20}) by
\begin{equation} \mathbb{M}_\psi (p) =
    \begin{cases}
    \int_0^\infty t^{p-1}(\psi(t)-\psi(0))\,dt, & p\in (-1,0),
    \\
    \int_0^\infty t^{p-1}\psi(t)\,dt, & p>0 \ \text{ such that } \ t^{p-1}\psi(t)\in L^1(\R^+).
    \end{cases}
    \label{eq:Mel}
\end{equation}
Clearly, the Mellin transform is piece-wise continuous. The relation between the $\mu$-weighted $(\hid,p)$ radial mean bodies and the Mellin transform is summarized in the following result.

\begin{proposition}\label{prop-4.3-1}
    Let $\hid,n\in\N$ be fixed. Suppose that $\mu$ is a Borel measure on $\Rn$ with support containing $K\in\conbod$. Then, for each $\bar{\theta}\in \S,$
    \begin{equation}
\begin{split}
\rho_{R^\hid_{p,\mu} K}(\bar{\theta})^p&=p\mathbb{M}_{\frac{g_{\mu,\hid}(K,r\bar\theta)}{\mu(K)}}(p)
\\
&=\begin{cases}
    \frac{p}{\mu(K)}\int_0^{\rho_{D^\hid(K)}(\bar\theta)}g_{\mu,\hid}(K,r\bar\theta)r^{p-1}dr, & p>0,
    \\
    \frac{p}{\mu(K)}\int_{0}^{\infty} \left(g_{\mu,\hid}(K,r\bar\theta)-\mu(K)\right) r^{p-1} d r, & p\in (-1,0).
\end{cases}
\end{split}
\label{radial_ell}
\end{equation}
The case $p=0$ again follows by continuity.
\label{p:other_form}
\end{proposition}

\begin{proof}
Observe that $x\in \cap_{i=1}^\hid(K+r\theta_i)$ if, and only if, $-r\theta_i\in K-x$ for every $i=1,\dots,\hid.$ But this is equivalent to $r\leq \rho_{K-x}(-\theta_i).$ Then, for $p>0$:
\begin{align*}
    \int_K&\left(\min_{i=1,\dots,\hid}\{\rho_{K-x}(-\theta_i)\}\right)^pd\mu(x)
    \\
    &=p\int_K\int_0^{\min_{i=1,\dots,\hid}\{\rho_{K-x}(-\theta_i)\}}r^{p-1}dr\,d\mu(x)
    \\
    &=p\int_{K}\int_0^\infty \chi_{\bigcap_{i=1}^\hid\{r\leq \rho_{K-x}(-\theta_i)\}}(r)r^{p-1}dr\,d\mu(x)
    \\
    &=p\int_0^\infty \int_{K}\chi_{\bigcap_{i=1}^\hid\{x\in (K+r\theta_i)\}}(x)\,d\mu(x)r^{p-1}dr
    \\
    &=p\int_0^\infty g_{\mu,\hid}(K,r\bar\theta)r^{p-1}dr \\ &=p\int_0^{\rho_{D^\hid(K)}(\bar\theta)} g_{\mu,\hid}(K,r\bar\theta)r^{p-1}dr.
\end{align*}
For $p\in (-1,0),$ we obtain
\begin{align*}
    \int_K&\left(\min_{i=1,\dots,\hid}\{\rho_{K-x}(-\theta_i)\}\right)^pd\mu(x) 
    \\
    &=-p\int_K\int_{\min_{i=1,\dots,\hid}\{\rho_{K-x}(-\theta_i)\}}^\infty r^{p-1}dr\,d\mu(x)
    \\
    &=-p\int_{K}\int_0^\infty \chi_{\bigcup_{i=1}^\hid\{r> \rho_{K-x}(-\theta_i)\}}(r)r^{p-1}dr\,d\mu(x)
    \\
    &=-p\int_0^\infty \int_{K}\chi_{\bigcup_{i=1}^\hid\{x\notin (K+r\theta_i)\}}(x)\,d\mu(x)r^{p-1}dr
     \\
    &=-p\int_0^\infty \mu\left(K\setminus \left(\bigcap_{i=1}^\hid (K+r\theta_i)\right)\right)r^{p-1}dr
    \\
    &=-p\int_0^\infty (\mu(K)-g_{\mu,\hid}(K,r\bar\theta))r^{p-1}dr.
\end{align*}
Inserting the definition of the Mellin transform, \eqref{eq:Mel}, yields the claim.
\end{proof}

It can be checked from Proposition~\ref{p:radial_ball}, Lemma~\ref{l:higher_covario},  and  Proposition~\ref{p:other_form} that, if $\mu$ is $s$-concave for $s \ge 0$ and $p \geq 0$, then  $R^\hid_{p,\mu} K$ is an $n\hid$-dimensional convex body containing the origin in its interior. However, for $p \in (-1,0)$, it is not known whether $R^\hid_{p,\mu} K$ is convex, even when $\mu$ is $s$-concave (or, more specifically, when $\mu$ is the Lebesgue measure) and $\hid = 1$. Observing from integration by parts, and the differentiability of $g_{\mu,\hid}(K,r\bar\theta)$ almost everywhere (as a function in $r$ on its support), we obtain that, for all $p\in (-1,0)\cup(0,\infty)$,
\begin{equation}\rho_{R^\hid_{p,\mu} K}(\bar{\theta})^p=-\frac{1}{\mu(K)}\int_0^{\rho_{D^\hid (K)}(\bar\theta)}(g_{\mu,\hid}(K,r\bar\theta))^\prime r^{p}dr,
\label{eq:beset_radial_form}
\end{equation}
where the derivative is in $r$ (notice since $g_{\mu,\hid}(K,r\bar\theta)$ is decreasing, the function $-(g_{\mu,\hid}(K,r\bar\theta))^\prime$ is positive). Clearly,   $R^\hid_{p,\mu} K \to \{o\}$ as $p\to (-1)^+$. In view of this, we need to renormalize $R^\hid_{p,\mu} K$ by $\left(p+1\right)^{1/p}R^\hid_{p,\mu} K$, which provides a wonderful relation between $R^\hid_{p,\mu} K$ and $\PP_\mu K$ as $p\to (-1)^+$. 

\begin{proposition}
\label{p:radial_body_limit}
    Fix $\hid,n\in\N$. Let $K\in \conbod$ be a convex body and $\mu\in \mathcal{M}(K).$ Then,  $$\lim_{p\to (-1)^+}\left(p+1\right)^{1/p}R^\hid_{p,\mu} K = \mu(K) \PP_\mu K.$$
\end{proposition}

\begin{proof} For $p\in (-1,0),$ let $p=-s$. Then $s\in (0,1)$ and by  \eqref{eq:beset_radial_form}, for each $\bar{\theta}\in \S,$
\begin{equation}(1-s)\rho_{R^\hid_{p,\mu} K}(\bar{\theta})^{-s}=(1-s)\int_0^{\rho_{D^\hid (K)}(\bar\theta)}\frac{(-g_{\mu,\hid}(K,r\bar\theta))^\prime}{\mu(K)} r^{-s}dr.
\label{eq_almost-1}
\end{equation} It has been proved in \cite{HL24} (see also \cite[Lemma 4] {HL22})
that if $\varphi:[0, \infty) \rightarrow[0, \infty)$ is a measurable function with $\lim _{t \rightarrow 0^{+}} \varphi(t)=$ $\varphi(0)$ and such that $\int_0^{\infty} t^{-s_0} \varphi(t) \, dt<\infty$ for some $s_0 \in(0,1)$, then \begin{align}
   \lim _{s \rightarrow 1^{-}}(1-s) \int_0^{\infty} t^{-s} \varphi(t)\,dt =\varphi(0).  \label{limit-s->1}
\end{align}
This, together with  \eqref{eq_almost-1},   Theorem~\ref{t:variationalformula} and Definition \ref{def:meas_bodies}, yields that,  for every $\bar\theta\in\S$, 
\begin{equation}\lim_{s\rightarrow 1^-} (1-s)\rho_{R^\hid_{p,\mu} K}(\bar{\theta})^{-s}= \frac{(-g_{\mu,\hid}(K,r\bar\theta))^\prime}{\mu(K)} \bigg|_{r=0^+} =\frac{h_{\P_\mu K}(\bar{\theta})}{\mu(K)}.
\label{eq_almost-1--11}
\end{equation} This further implies, setting $p=-s$, that 

\begin{equation}
\label{eq:-1_body}
\lim_{p\to (-1)^+}\left(p+1\right)^{1/p}\rho_{R^\hid_{p,\mu} K}(\bar\theta)=\mu(K) \rho_{\PP_\mu K}(\bar\theta).\end{equation} This completes the proof. 
\end{proof}
 
Thus, the \textit{shape} of $R^{\hid}_{p,\mu} K$ approaches that of $\mu(K) \PP_\mu K$ as $p\to (-1)^+.$
On the other hand, applying  H\"{o}lder's inequality to  \eqref{eq:beset_radial_form}  with respect to the probability measure $-\mu(K)^{-1}g_{\mu,\hid}(K,r\bar\theta)^\prime\,dr,$  if $\mu\in\mathcal{M}(K)$ for a fixed $K\in\conbod$, the following holds for $-1<p\leq q <\infty$: 
$$R^\hid_{p,\mu} K \subseteq R^\hid_{q,\mu} K \subseteq D^\hid(K).$$

We now show how this chain of inclusions may be reversed under an $F$-concavity assumption.
\begin{theorem}
\label{t:radial_F_inclusions}
Let $K\in\conbod$ be a convex body. Suppose that $F:[0,\mu(K))\to [0,\infty)$ is a continuous, increasing, and invertible function. Let $\mu$ be a finite Borel measure which is  $F$-concave on the class of convex subsets of $K$. Then, for $-1<p\le q <\infty$, one has
$$D^\hid (K)\subseteq C(q,\mu,K) R^\hid_{q,\mu}K \subseteq C(p,\mu,K) R^\hid_{p,\mu} K \subseteq \frac{F(\mu(K))}{F^\prime(\mu(K))}\PP_\mu K,$$
where $C(p,\mu,K)=$
$$\begin{cases}
    \left(\frac{p}{\mu(K)}\!\int_{0}^{1}\! F^{-1}\left[F(\mu(K))(1-t)\right]t^{p-1}dt\right)^{-\frac{1}{p}} &  \ \text {for}\   p>0, \\
    \left(\frac{p}{\mu(K)}\!\int_{0}^{1}\! (F^{-1}\left[F(\mu(K))(1-t)\right]-\mu(K))t^{p-1}dt+1\right)^{-\frac{1}{p}} & \  \text {for}\  p\in (-1,0),
    \end{cases}
    $$
and, for the last set inclusion, we additionally assume that $\mu\in\mathcal{M}(K)$ and that $F(x)$ is differentiable at the value $x=\mu(K).$
The equality conditions are the following:
\begin{enumerate}
    \item For the first two set inclusions there is equality of sets if, and only if, $F(0)=0$ and $F\circ g_{\mu,\hid}(K,x)=F(\mu(K))\ell_{D^\hid (K)}(x),$ where $\ell_K$ is the \textit{roof function} of $K$ defined in \eqref{eq:lk}. 
    \item For the last set inclusion, the sets are equal if, and only if, $F\circ g_{\mu,\hid}(K;x)=F(\mu(K))\ell_C(x)$ with $C=\frac{F(\mu(K))}{F^\prime (\mu(K))}\Pi^\circ_{\mu}K.$
\end{enumerate}
\end{theorem}
 
\begin{proof}
Let $G(p):=C(p,\mu,K) \rho_{R^\hid_{p,\mu}K}(p)$. From Berwald's inequality for $F$-concave measures in \cite{LP23}, this function is non-increasing in $p$, which establishes the first three set inclusions. For the last set inclusion, we first rewrite 
$$G(p)=\frac{C(p,\mu,K)}{(p+1)^{1/p}} (p+1)^{1/p}\rho_{R^\hid_{p,\mu}K}(\theta).$$
Therefore, from Proposition~\ref{p:radial_body_limit}, it suffices to show that, as $p\to -1,$ $$\frac{C(p,\mu,K)}{(p+1)^{1/p}} \to \frac{F(\mu(K))}{F^\prime(\mu(K))\mu(K)}.$$
Indeed, from integration by parts,  for all $p\in (-1,0)\cup(0,\infty),$ one gets
$$C(p,\mu,K)=\left(\frac{F(\mu(K))}{\mu(K)}\right)^{-\frac{1}{p}}\left(\int_0^1\left[F^\prime \left(F^{-1}[F(\mu(K))(1-t)]\right)\right]^{-1}t^pdt\right)^{-\frac{1}{p}}.$$
Therefore, the result follows from \eqref{limit-s->1}. 
\end{proof}

 As a byproduct, we obtain Theorem~\ref{t:set_s_con}, which we reproduce for the convenience of the reader.

 \vskip 2mm \noindent{\bf Theorem \ref{t:set_s_con}.} {\em 
Let $K\in\conbod$ and $\hid\in\N$ be fixed. Suppose that $\mu$ is an $s$-concave Borel measure, $s>0,$ on convex subsets of  $K$. Then, for $-1< p\leq q < \infty$, one has
$$D^\hid (K)\subseteq \binom{\frac{1}{s}+q}{q}^{\frac{1}{q}} R^\hid_{q,\mu}K \subseteq \binom{\frac{1}{s}+p}{p}^{\frac{1}{p}} R^\hid_{p,\mu} K\subseteq \frac{1}{s}\mu(K)\PP_{\mu}K,$$
where the last inclusion holds if $\mu\in\mathcal{M}(K)$.

\vskip 2mm  There is an equality in any set inclusion if, and only if, $$g_{\mu,\hid}^s(K,x)=\mu(K)^s\ell_{D^\hid (K)}(x).$$ If $\mu$ is a locally finite and regular Borel measure, i.e., $s$-concave on compact subsets of its support, then $s\in (0,1/n]$ and equality occurs if, and only if, $K$ is a $n$-dimensional simplex, $\mu$ is a positive multiple of the Lebesgue measure, and $s = \frac{1}{n}$. }

\begin{proof} Applying Theorem~\ref{t:radial_F_inclusions} to $F(x)=x^s$,  one gets,  when $p>0,$
$$C(p,\mu,K)=\left(p\int_{0}^{1} (1-t)^{1/s}t^{p-1}dt\right)^{-\frac{1}{p}}=\left(\frac{p\Gamma(\frac{1}{s}+1)\Gamma(p)}{\Gamma(\frac{1}{s}+p+1)}\right)^{-\frac{1}{p}},$$
and similarly for $p\in (-1,0).$ The equality conditions from Theorem~\ref{t:radial_F_inclusions} yield that $g_{\mu,\hid}^s(K;x)$ is an affine function along rays for $x\in D^\hid (K)$. If $\mu$ is a locally finite and regular measure on compact sets, then one must have $s\in (0,1/n]$ from Borell's classification of concave measures; for such $s$-concave measures, Proposition~\ref{p:simp} above shows that  $K$ is a $n$-dimensional simplex, the density of $\mu$ is constant, and $s = \frac{1}{n}$.
\end{proof}

Note that it is assumed in Theorem \ref{t:radial_F_inclusions} that $F \ge 0$. Without this assumption,   $C(p,\mu,K)$ may tend to $0$ as $p\to \infty,$ and so $C(p,\mu,K)R^\hid_{p,\mu}K$ will tend to the origin; hence, the first set inclusion may be lost. However, the assumption that $F$ is nonnegative fails, for instance, in the important case of log-concave functions. Thus, we give a result for possibly negative $F$ as well, which is slightly weaker:

\begin{theorem}
\label{t:set_inclu_log}
Fix $\hid\in\N$. Suppose a Borel measure $\mu$ on $\R^n$ with density is finite on some $K\in\conbod$ and $Q$-concave, where $Q:(0,\mu(K)]\to (-\infty,\infty)$ is an increasing and invertible function. Then, for $-1<p\le q <\infty$, one has
$$C_Q(q,\mu,K)R^\hid_{q,\mu}K \subset C_Q(p,\mu,K)R^\hid_{p,\mu}K \subset \frac{1}{Q^\prime(\mu(K))}{\PP_{\mu}K},$$
where $C_Q(p,\mu,K)=$ $$\begin{cases}
    \left(\frac{p}{\mu(K)}\int_0^\infty Q^{-1}\left[Q(\mu(K))-t\right]t^{p-1} dt\right)^{-\frac{1}{p}} &\   \text {for } \ \ p>0, \\
    \left(\frac{p}{\mu(K)}\int_{0}^{\infty}t^{p-1} (Q^{-1}\left[Q(\mu(K)-t)\right]-\mu(K))\,dt\right)^{-\frac{1}{p}} & \text\   {for}  \ p\in (-1,0),
    \end{cases}$$    
and, for the second set inclusion, we additionally assume that $\mu\in\mathcal{M}(K)$ and that $Q(x)$ is differentiable at the value $x=\mu(K)$. In particular, if $\mu$ is log-concave, then
$$\frac{1}{\Gamma\left(1+q\right)^{\frac{1}{q}}}R^\hid_{q,\mu}K \subset \frac{1}{\Gamma\left(1+p\right)^{\frac{1}{p}}}R^\hid_{p,\mu}K \subset \mu(K){\PP_{\mu}K},$$
where $\lim_{p\to 0} \frac{1}{\Gamma\left(1+p\right)^{\frac{1}{p}}}R^\hid_{p,\mu}K$ is interpreted via continuity.
\end{theorem}

 \begin{proof}
The first inclusion follows from the second case of Berwald's inequality for measures, established in \cite{LP23}. For the second inclusion, we can assume without loss of generality that $p>0$. Then, for every $\bar\theta\in\S$, one has, by Lemma~\ref{l:concave} applied to $f=Q\circ g_{\mu,m}(K,\cdot)$,

$$0\le g_{\mu,\hid}(K;r\bar\theta)\le Q^{-1}\left[Q(\mu(K))\left(1-\frac{Q^\prime(\mu(K))}{Q(\mu(K))}\frac{r}{\rho_{\PP_{\mu} K}(\bar\theta)}\right)\right].$$
Since $Q(\mu(K))$ may possibly be negative, we shall leave $Q(\mu(K))$ inside the integral and obtain, with the help of Proposition \ref{prop-4.3-1}: 
\begin{align*}
 &\rho^p_{R^\hid_{p,\mu} K}(\bar\theta)
 \\
 &=\frac{p}{\mu(K)}\int_{0}^{\rho_{D^\hid (K)}(\bar\theta)} g_{\mu,\hid}(K;r\bar\theta) r^{p-1} d r 
\\
&\leq \frac{p}{\mu(K)}\int_{0}^{\rho_{D^\hid (K)}(\bar\theta)} Q^{-1}\left[Q(\mu(K))\left(1-\frac{Q^\prime(\mu(K))}{Q(\mu(K))}\frac{r}{\rho_{\PP_{\mu} K}(\bar\theta)}\right)\right]r^{p-1} d r.
\\
&=\left(\frac{\rho_{\PP_{\mu} K}(\bar\theta)}{Q^\prime(\mu(K))}\right)^p\frac{p}{\mu(K)}\int_{0}^{Q^\prime(\mu(K))\frac{\rho_{D^\hid (K)}(\bar\theta)}{\rho_{\PP_{\mu} K}(\bar\theta)}} Q^{-1}\left[Q(\mu(K))-t\right]t^{p-1}dt.
\end{align*}
Consequently, 
$C_Q(p,\mu,K)\rho_{R^\hid_{p,\mu} K}(\bar\theta) < \frac{1}{Q^\prime(\mu(K))}\rho_{\PP_{\mu} K}(\bar\theta),$
which yields the result.
\end{proof}

\section{Generalization of Chord Integral Inequalities}
\label{sec:genvol}

In this section, we demonstrate how Lemma~\ref{t:chak} serves as a prototype of a more general theorem, which is of independent interest.  We require some background notions.

Let $\mathbb{G}: (0, \infty)\times \s \rightarrow (0, \infty)$ be a continuous function such that $\mathbb{G}(0,\theta)=0$ for almost every $\theta\in \s$ and $$\mathbb{G}_t(t, \theta)=\frac{\partial \mathbb{G}(t, \theta)}{\partial t}$$ is continuous on $(0, \infty)\times \s.$ Let $G=\mathbb{G}_t$ and assume that $G: (0, \infty)\times \s \rightarrow (0, \infty)$ is a continuous function such that, for any $R\in (0, \infty)$, \begin{align}
   \int_{(0, R)\times \s}G(r, \theta)\,dr\,d\theta<\infty. \label{assump-G}
\end{align} The general dual volume of a star body $L\subset \Rn$ \cite{GHX19, GHXY20} can be formulated by $$\widetilde{V}_{\mathbb{G}}(L)=\frac{1}{n}\int_{\s}\mathbb{G}(\rho_L(\theta), \theta)\,d\theta=\frac{1}{n}\int_{\s} \int_0^ {\rho_L(\theta)} G(r, \theta)\,dr\,d\theta.$$

The goal of this section is to generalize Lemma \ref{t:chak} to the setting of the general dual volume. 

\begin{theorem}[Two Chord Integral Inequalities for Generalized Volume]
\label{ber_type}
 Suppose $f: [0, \infty)\times \s \rightarrow [0, \infty)$ is a non-negative function supported on $L\in\conbod_0$ such that for every $\theta\in\s,$ $f(t, \theta)$ is concave on $t\in (0, \rho_L(\theta))$. Let $h: [0, \infty)\rightarrow [0, \infty)$ be an increasing, non-negative, differentiable function. Let $G: (0, \infty)\times \s \rightarrow (0, \infty)$ be a function such that \eqref{assump-G} holds. Fix some $\alpha >-1$.
 The following statements hold. 
 \begin{enumerate}
     \item If $G(ur, \cdot)\geq u^{\alpha}G(r, \cdot)$ for $u\in[0,1]$, $r>0,$ one has \begin{equation}\int_{\s}\int_0^{\rho_L(\theta)}h(f(r, \theta))G(r, \theta)\,dr\,d\theta \geq n \beta_{\alpha} \widetilde{V}_{\mathbb{G}}(L),
     \label{globai-increasing--0}
     \end{equation}
     where
 \begin{align*}
\beta_{\alpha} = \inf_{\theta\in \s}\left[ (\alpha+1)  \int_0^{1} h\left(f(0, \theta)\tau\right)(1-\tau)^{\alpha}   d\tau\right].
\end{align*}

There is equality in \eqref{globai-increasing--0} if, and only if,  $G(r, \cdot)$ is homogeneous of degree $\alpha$, $f(0, \theta)$ is a constant on $\s,$ and $$f(r, \theta)=f(0, \theta)\left(1-(\rho_L(\theta))^{-1}r\right)=f(0, \theta)\aff L(r\theta),$$ that is, $f$ is affine on each ray. 

     \item Suppose that $G(ur, \cdot)\leq u^{\alpha}G(r, \cdot)$ for $u\in[0,1]$, $r>0$ and $$\max_{0\leq r\leq \rho_L(\theta)} f(r, \theta)=f(0, \theta)$$ for each $\theta\in \s$. Let $$\Omega_f:=\left\{\theta\in\s: \frac{\partial f(r, \theta)}{\partial r}\bigg|_{r=0^+}=0\right\}$$ and let
     \begin{align*}
     \beta_{b}=  \sup_{\theta\in \s\setminus \Omega_f}\left[ (\alpha+1)  \int_0^{1} h\Big(f(0, \theta)\tau\Big)(1-\tau)^{\alpha}   d\tau\right].
 \end{align*} 

     Then,
     \begin{equation}
     \begin{split}
    \int_{\s}&\int_0^{\rho_L(\theta)}h(f(r, \theta))G(r, \theta)\,dr\,d\theta 
    \\
    &\leq \beta_{b}\int_{\s\setminus\Omega_f} \int_0^{\rho_{\widetilde{L}(\theta)}}G(r,\theta)\,dr\,d\theta  \\
&\quad\quad+\int_{\Omega_f}\int_0^{\rho_L(\theta)}h(f(0,\theta))G(r,\theta)\,dr\,d\theta,
    \end{split}
    \label{globai-increasing--1}
\end{equation}
where $\rho_{\widetilde{L}}(\theta)=-\left( \frac{\partial f(r, \theta)}{\partial r}\big|_{r=0^+}\right)^{-1}f(0, \theta).$
If $\Omega_f=\emptyset$ and $\frac{\partial f(r, \theta)}{\partial r}\big|_{r=0^+}$ is also continuous for almost all $\theta\in \s,$ then $\rho_{\widetilde{L}}(\theta)$ is the radial function of some star body $\widetilde{L}$ that contains $L$. In this instance, \eqref{globai-increasing--1} becomes
\begin{equation}
    \int_{\s}\int_0^{\rho_L(\theta)}h(f(r, \theta))G(r, \theta)\,dr\,d\theta \leq n\beta_{b}\widetilde{V}_{\mathbb{G}}(\widetilde{L}).
    \label{globai-increasing--2}
\end{equation}
 \end{enumerate}

There is equality in \eqref{globai-increasing--1} if, and only if: 
\begin{enumerate}
    \item for $\theta\in\s\setminus\Omega_f$, one has that $\rho_{\widetilde{L}}(\theta)=\rho_{\supp(f)}(\theta),$ $f(0, \theta)$ is a constant on $\s,$ and $$f(r, \theta)=f(0, \theta)\left(1-(\rho_L(\theta))^{-1}r\right)=f(0, \theta)\aff L(r\theta),$$ that is $f$ is affine on each ray, and, for every $r>0,$ $G(r,\cdot)$ has homogeneity of degree $\alpha$; 
    \item for almost every $\theta\in\Omega_f,$ one has that $f(r,\theta)=f(0,\theta)$ for every $r\in [0,\rho_{L}(\theta)].$
\end{enumerate}
If $\Omega_f=\emptyset$, then equality in \eqref{globai-increasing--2} yields $L=\widetilde{L}.$
\end{theorem}

\begin{proof}
We first show inequality \eqref{globai-increasing--0}. For $L\in \conbod_0,$ let 
\begin{align}
 H(L)=\int_{\s}\int_0^{\rho_L(\theta)}h(f(r, \theta))G(r, \theta)\,dr\,d\theta. \label{def-HL}   
\end{align}   Note that the function $H(L)$ is well defined by \eqref{assump-G}. 
For $\theta\in \s$, let $g_{\theta}: [0, \rho_L(\theta))\times \s \rightarrow \mathbb{R}^{+}$ be given by
$$
g_{\theta} (r)=f(0, \theta)\aff L(r\theta)=f(0, \theta) \Big(1-
\frac{r}{\rho_L(\theta)}\Big).
$$ Since $f(t, \cdot)$ is concave on $t\in [0, \infty)$, one has, \begin{align*}f(r, \theta)=f\Big(\frac{r}{\rho_L(\theta)}\rho_L(\theta), \theta\Big)&\geq  \frac{r}{\rho_L(\theta)} f\Big(\rho_L(\theta), \theta\Big)+ f(0, \theta) \left(1-\frac{r}{\rho_L(\theta)}\right)
\\
&\geq f(0, \theta) \Big(1-\frac{r}{\rho_L(\theta)}\Big).
\end{align*}
Since $h$ is increasing, one has $h(f(r, \theta)) \geq h(g_{\theta}(r)).$ Therefore, the monotonicity of the integral implies that
\begin{equation}H(L)   \geq \int_{\s} \int_{0}^{\rho_{L}(\theta)} h(g_{\theta}(r)) G(r, \theta)   d r d \theta.
\label{e:ch1}
\end{equation}
 For $y\in [0, 1],$ let
  \begin{align}
      \psi(y):= \beta ^\prime &\int_{\s}\int_0^{y\rho_L(\theta)} G(r, \theta)\, dr\,d\theta  
      \\
      &- \int_{\s}\int_0^{y\rho_L(\theta)} h\left(f(0,\theta)\left[1 - \frac{r}{y\rho_L(\theta)}\right]\right)G(r, \theta)  \,dr\,d\theta, \label{def-psi-1}
  \end{align} where  $\beta^\prime > 0$ is a constant, independent of the direction $\theta$ and the function $G$, chosen such that $\psi(y) \leq 0$ on $y\in [0, 1]$.

We now show that $\beta^\prime$ exists and that $\beta^\prime=\beta_{\alpha}$ for almost every $\theta\in \s$. By hypothesis, we have that $G(r, \theta)$ is continuous on $(0,\infty)\times \s$ and bounded in the sense of \eqref{assump-G}. Additionally, $h$ is integrable on each segment $(0, s] \subset (0,\infty).$ Indeed, since $h$ is an increasing function, it is piece-wise continuous almost everywhere, and therefore it is dominated on each segment by an integrable function. Consequently, we may assert that $\psi (y) \to 0$ as $y \to 0^+$. Since $\psi$ is absolutely continuous on each $[a,b] \subset (0,y]$, $\psi$ may be represented by 
 \[
\psi(y) = \psi(a) + \int_a^y \psi'(s)\, ds.
\]

In order to have $\beta^\prime >0$ such that $\psi (y) \leq 0$, it suffices for $\beta^\prime$ to be selected so that $\psi ^\prime(y) \leq 0$ for almost every $y \in (0,1]$. Differentiation of $\psi $ yields the representation
\begin{align*}
\psi^\prime(y) &= \beta^\prime I_{G, L}(y)-h(0) I_{G, L}(y)
\\
&\quad-\int_{\s}f(0, \theta) \int_0^{y\rho_{L}(\theta)} h^\prime\left(f(0, \theta)\left[1 - \frac{r}{y\rho_{L}(\theta)}\right]\right)\frac{rG(r, \theta) }{y^2\rho_{L}(\theta)} \,dr\,d\theta, 
\end{align*} where, from the positivity of $G$, we have $$I_{G, L}(y)=\int_{\s} G(y\rho_{L}(\theta), \theta)\rho_{L}(\theta)\,d\theta>0.$$ This further yields that 
\begin{align*}
&\beta ^\prime \leq h(0)
\\
&+ I_{G, L}(y)^{-1}\int_{\s}f(0, \theta) \int_0^{y\rho_{L}(\theta)} h^\prime\left(f(0, \theta)\left[1 - \frac{r}{y\rho_{L}(\theta)}\right]\right)\frac{rG(r, \theta) }{y^2\rho_{L}(\theta)} \,dr\,d\theta,
\end{align*}
or equivalently, applying the change of variables $r=uy\rho_L(\theta)$, we see that $\beta^\prime$ must satisfy
\begin{equation}
\begin{split}
    &\beta ^\prime \leq h(0)
    \\
    &+ I_{G, L}(y)^{-1}\!\!\int_{\s}\!\!f(0, \theta)\!\int_0^{1}\! h^\prime\!\left(\!f(0, \theta)\!\left[1\! -\!u\right]\!\right)\!uG(uy\rho_L(\theta), \theta) \rho_{L}(\theta)  dud\theta.
    \end{split}
 \label{determine-beta}
\end{equation}

By using $G(ur, \cdot)\geq u^{\alpha}G(r, \cdot)$ for $u\in[0,1]$, $r>0$ and some constant $\alpha>-1$, we can get, by letting $\tau=1-u$ and by using integration by parts, 
\begin{align}
    &\int_{\s}f(0, \theta) \int_0^{1} h^\prime\left(f(0, \theta)\left[1 -u\right]\right)uG(uy\rho_L(\theta), \theta) \rho_L(\theta) dud\theta \nonumber \\
     &\geq \int_{\s}f(0, \theta) \int_0^{1} h^\prime\left(f(0, \theta)\left[1 -u\right]\right)u^{\alpha+1}G(y\rho_L(\theta), \theta) \rho_L(\theta) \,du\,d\theta \nonumber \\ 
    &= \int_{\s}f(0, \theta) \int_0^{1} h^\prime\left(f(0, \theta)\tau\right)(1-\tau)^{\alpha+1}G(y\rho_L(\theta), \theta) \rho_L(\theta)\, d\tau d\theta  \nonumber \\ 
    &= \int_{\s} \left(f(0, \theta) \int_0^{1} h^\prime\left(f(0, \theta)\tau\right)(1-\tau)^{\alpha+1}   d\tau \right) G(y\rho_L(\theta), \theta)\rho_L(\theta) d\theta \nonumber \\ 
    &= \int_{\s} \!\left(-h(0)+ (\alpha+1) \int_0^{1}\! h\left(f(0, \theta)\tau\right)(1-\tau)^{\alpha}   d\tau \right) G(y\rho_L(\theta), \theta) \rho_L(\theta)d\theta. \label{add-11-11}
\end{align} In view of \eqref{determine-beta}, one can just let  
\begin{align}
 \beta ^\prime =  \inf_{\theta\in \s}\left[ (\alpha+1)  \int_0^{1} h\Big(f(0, \theta)\tau\Big)(1-\tau)^{\alpha} \, d\tau\right],       \label{add-22-22} \end{align} which satisfies our requirement.   
Thus, \eqref{def-HL}, \eqref{e:ch1} and \eqref{def-psi-1} imply that $\psi(1)\leq 0$ and then 
\begin{align*}
\int_{\s}\int_0^{\rho_L(\theta)}&h(f(r, \theta))G(r, \theta)\,dr\,d\theta 
\\
&\geq  \beta_{\alpha} \int_{\s} \int_0^{\rho_L(\theta)}  G(r,\theta), dr\,d\theta =n \beta_{\alpha}\widetilde{V}_{\mathbb{G}}(L).
\end{align*} Furthermore, we see that equality occurs when $G(r, \cdot)$ is $\alpha$-homogeneous, $f(0, \theta)$ is a constant on $\s,$ and $f(r, \theta)=f(0, \theta)\left(1-(\rho_L(\theta))^{-1}r\right)=f(0, \theta)\aff L(r\theta),$ that is $f$ is affine on each ray; these conditions also are the necessary conditions to have the equality. To see the latter one, assume the equality to be true. Then inequalities \eqref{e:ch1} and  \eqref{add-11-11} both become equalities and hence   $f(r, \theta)=f(0, \theta)\aff L(r\theta)$ and $G(r, \cdot)$ has homogeneity of degree $\alpha$, respectively. Moreover,  \eqref{add-11-11} and \eqref{add-22-22} yield that $f(0, \theta)$ must be a constant over $\theta\in \s$, due the continuity of $f(0, \theta)$ and the monotonicity of $h.$ This completes the proof for the first inequality and its equality characterization.

Now let us show \eqref{globai-increasing--1}, i.e. prove the case when  $G: (0, \infty)\times \s \rightarrow (0, \infty)$ is a function such that  \eqref{assump-G} holds, and   $G(ur, \cdot)\leq u^{\alpha}G(r, \cdot)$ for $u\in[0,1]$, $r>0$ and some constant $\alpha>-1$.  The hypotheses $\max_{0\leq r\leq \rho_L(\theta)} f(r, \theta)=f(0, \theta)$ for all $\theta\in \s$ and the concavity of $f(r, \cdot)$ on $r\in (0,\infty)$ yield that for almost every $\theta\in\s,$
$$ \frac{\partial f(r, \theta)}{\partial r}\bigg|_{r=0^+}\leq 0.$$
Consequently, we have that, for $\theta\in\Omega_f$ and $r\in [0,\rho_L(\theta)]$ that
$$0\leq f(r, \theta) \leq f(0, \theta )\left[1+\frac{ \frac{\partial f(r, \theta)}{\partial r}\big|_{r=0^+}}{f(0, \theta)}r\right]=f(0, \theta ).$$
Therefore,
\begin{align*}\int_{\Omega_f}\int_0^{\rho_L(\theta)}h(f(r, \theta))G(r, \theta) \,dr\,d\theta
&\leq \int_{\Omega_f}\int_0^{\rho_L(\theta)}h\left(f(0,\theta)\right)G(r, \theta)\,dr\,d\theta.\end{align*}

On the other hand, for $\theta\in\s\setminus\Omega_f$ and $r\in [0,\rho_L(\theta)]$ we obtain
$$0\leq f(r, \theta) \leq f(0, \theta )\left[1+\frac{ \frac{\partial f(r, \theta)}{\partial r}\big|_{r=0^+}}{f(0, \theta)}r\right]=f(0, \theta )\left[1-\frac{r}{\rho_{\widetilde{L}}(\theta)}\right].$$ Since this is true for all $r\in [0,\rho_L(\theta)]$, one gets $\rho_L(\theta) \leq \rho_{\widetilde{L}}(\theta)$ for $\theta\in\s\setminus \Omega_f$. Hence,  one has
\begin{align*}\int_{\s\setminus \Omega_f}&\int_0^{\rho_L(\theta)}h(f(r, \theta))G(r, \theta) \,dr\,d\theta
\\
&\leq \int_{\s\setminus \Omega_f}\int_0^{\rho_L(\theta)}h\left(f(0,\theta) \left[1-\frac{r}{\rho_{\widetilde{L}}(\theta)}\right]\right)G(r, \theta)\,dr\,d\theta
\\
&\leq \int_{\s\setminus \Omega_f}\int_0^{\rho_{\widetilde{L}}(\theta)} h\left(f(0,\theta)\left[1-\frac{r}{\rho_{\widetilde{L}}(\theta)}\right]\right)G(r,\theta)\,dr\,d\theta.\end{align*}
The proof for $\theta\in\s\setminus\Omega_f$ then follows similarly to the first case (by changing the direction of the inequalities and replacing the infimum by the supremum).
In the case that $\Omega_f=\emptyset,$ we remark that $\rho_L(\theta) \leq \rho_{\widetilde{L}}(\theta)$ implies $L\subseteq \widetilde{L}.$ If there is equality in the inequality in this instance, then  $\rho_L(\theta) = \rho_{\widetilde{L}}(\theta)$ for every $\theta\in\s,$ yielding $L=\widetilde{L}$; the remaining equality conditions following analogously to the previous inequality.
\end{proof}
\begin{remark}
    The following variant of \eqref{globai-increasing--0} holds, whose proof is the same. Let $H$ be a subspace of $\R^n$ with dimension $j$. Then, for an explicit $\beta_{\alpha,H}$
    \begin{equation}\int_{\s\cap H}\int_0^{\rho_{L\cap H}(\theta)}h(f(r, \theta))G(r, \theta)\,dr\,d\theta \geq j \beta_{\alpha,H} \widetilde{V}_{\mathbb{G}}(L\cap H).
     \label{globai-increasing--0_0}
     \end{equation}
     One can apply this result to obtain the following. Firstly, consider a subspace of $\R^n$ with the following structure: 
    $H=H_1\otimes \cdots \otimes H_m$, where each $H_i$ is dimension $n_i \in \{1,\dots,n-1,n\}$. Next, apply \eqref{globai-increasing--0_0} with $L=D^m(K)$, $\mu=\mu_1\times \cdots \times \mu_m$, with each $\mu_i$ a radially decreasing measure on $\R^{n}$, and $G(r,\theta)=\phi(r\theta)r^{nm-1}$, with $\phi$ the density of $\mu$. Then, one obtains \cite[Theorem 3]{Ro20}.
\end{remark}

Observe that, in Theorem~\ref{ber_type}, if $\Omega_f=\emptyset$ and $f(0, \theta)$ is a constant function on $\theta\in \s$, then $\beta_{\alpha}=\beta_b.$ In particular,  $f(r, \theta)=f_1(r\theta)$ for some function $f_1:\R^n\rightarrow [0, \infty),$ and so $$\beta_{\alpha}=\beta_b=  (\alpha+1)  \int_0^{1} h\left(f_1(0)\tau\right)(1-\tau)^{\alpha}   d\tau.$$

In Theorem~\ref{ber_type}, we have actually proven something stronger than asserted, since bounds were done on the integral over $\R$ and not over $\s.$ We outline these local versions as a corollary.
\begin{corollary}
\label{corr:ber_type}
      Suppose $f: [0, \infty)\times \s \rightarrow [0, \infty)$ is a non-negative function supported on $L\in\conbod_0$ such that $f(t, \cdot)$ is concave on $t\in [0, \infty)$,  $h: [0, \infty)\rightarrow [0, \infty)$ is an increasing, non-negative, differentiable function, and  $G: (0, \infty)\times \s \rightarrow (0, \infty)$ is a function such that \eqref{assump-G} holds. Fix some constant $\alpha >-1.$ The following statements hold. 
      \begin{enumerate}
      \item If $G(ur, \cdot)\geq u^{\alpha}G(r, \cdot)$ for $u\in[0,1]$, $r>0,$ then, for almost every $\theta\in\s,$
      \begin{align*} \int_0^{\rho_L(\theta)}&h(f(r, \theta))G(r, \theta)dr 
      \\
      & \geq \left(\int_0^{\rho_L(\theta)} G(r, \theta)\, dr\right) \cdot \left( (\alpha+1)  \int_0^{1} h\left(f(0, \theta)\tau\right)(1-\tau)^{\alpha}   d\tau\right) \\ & =\beta_{\alpha}(\theta)\left(\int_0^{\rho_L(\theta)} G(r, \theta)\, dr\right)=\beta_{\alpha}(\theta)\mathbb{G}(\rho_L(\theta), \theta); \end{align*}
      
          \item If $G(ur, \cdot)\leq u^{\alpha}G(r, \cdot)$ for $u\in[0,1]$, $r>0,$ and $$\max_{0\leq r\leq \rho_L(\theta)} f(r, \theta)=f(0, \theta),$$ then, for each $\theta\in \s\setminus \Omega_f$, with $\Omega_f$ defined as in Theorem~\ref{ber_type},
          \begin{align*} \int_0^{\rho_L(\theta)}& h(f(r, \theta))G(r, \theta)\,dr
          \\
          &\leq \left(\int_0^{\rho_{\widetilde{L}}(\theta)} G(r, \theta)\, dr\right) \cdot \left( (\alpha+1)  \int_0^{1} h\Big(f(0, \theta)\tau\Big)(1-\tau)^{\alpha}   d\tau\right) \\ &=\beta_{b}(\theta)\left(\int_0^{\rho_{\widetilde{L}}(\theta)} G(r, \theta)\, dr\right)=\beta_{b}(\theta)\mathbb{G}(\rho_{\widetilde{L}}(\theta), \theta). \end{align*} 
      \end{enumerate}
\end{corollary}

The results in Theorem \ref{ber_type} are based on the assumption \eqref{assump-G}, but a similar result can be obtained for those $G:(0, \infty)\times \s$ such that $G: (0, \infty)\times \s \rightarrow (0, \infty)$ is a continuous function, and for any $R\in (0, \infty)$, \begin{align}
   \int_{(R, \infty)\times \s}G(r, \theta)\,dr\,d\theta<\infty. \label{assump-G-inf}
\end{align} Indeed, under the above conditions on $G$ and $G(ur, \cdot)\geq u^{\alpha}G(r, \cdot)$ for $u\in[1, \infty)$ and $r>0$ with  some constant $\alpha<-1$ (note that $\alpha<-1$ is natural due to \eqref{assump-G-inf}),  if $\xi>0$ and $h: [0, \infty)\rightarrow [0, \infty)$ is an increasing, non-negative, differentiable function, then for any $y\in (0, \infty)$, one has, 
\begin{align}
      \int_{\s}\int_{y\rho_L(\theta)}^{\infty} & h\left(\xi \left[1 - \frac{y\rho_L(\theta)}{r}\right]\right)G(r, \theta)  \,dr\,d\theta
      \\
     &\geq  \widetilde{\beta_{\xi}}  \int_{\s}\int_ {y\rho_L(\theta)}^{\infty} G(r, \theta)\, dr\,d\theta
     \\
      &\geq n \widetilde{\beta_{\xi}} \widetilde{V}_{\widetilde{\mathbb{G}}}(yL);  \label{def-psi-2222-dec}
  \end{align} and its local version: for each $\theta\in \s$, 
  \begin{align}
     \int_ {y\rho_L(\theta)}^{\infty}  h\left(\xi \left[1 - \frac{y\rho_L(\theta)}{r}\right]\right)G(r, \theta) \, dr 
     \geq  \widetilde{\beta_{\xi}} \widetilde{\mathbb{G}}(y\rho_L(\theta), \theta), \label{local-22--22-dec}
  \end{align}
    where the constant $\widetilde{\beta_{\xi}}$ and the function $\widetilde{\mathbb{G}}$ are given by 
\begin{align*}
    \widetilde{\beta_{\xi}}&=-(\alpha+1)  \int_0^{1} h(\xi \tau)(1-\tau)^{-(2+\alpha)}   d\tau, \\ \widetilde{\mathbb{G}}(t, \theta)&=\int_t^{\infty} G(r, \theta)\,dr \ \ \ \ \ \mathrm{for\ all \ } (t, \theta)\in (0, \infty)\times \s,\\ \widetilde{V}_{\widetilde{\mathbb{G}}}(L)&=\frac{1}{n}\int_{\s} \widetilde{\mathbb{G}}(\rho_L(\theta), \theta)\, d\theta=\frac{1}{n}\int_{\s}\int_ {\rho_L(\theta)}^{\infty} G(r, \theta)\, dr\,d\theta .
\end{align*} Similar results hold for the case when  $G: (0, \infty)\times \s \rightarrow (0, \infty)$ is a function such that \eqref{assump-G-inf} holds, and $G(ur, \cdot)\leq u^{\alpha}G(r, \cdot)$ for $u\in[1, \infty)$ and $r>0$ with  some constant $\alpha<-1$. In this case, if $\xi>0$ and $h: [0, \infty)\rightarrow [0, \infty)$ is an increasing, non-negative, differentiable function, then for any $y\in (0, \infty)$, one has, 
 \begin{align}
      \int_{\s}\int_{y\rho_L(\theta)}^{\infty}& h\left(\xi \left[1 - \frac{y\rho_L(\theta)}{r}\right]\right)G(r, \theta)  \,dr\,d\theta
      \\
     &\leq  \widetilde{\beta_{\xi}}  \int_{\s}\int_ {y\rho_L(\theta)}^{\infty} G(r, \theta)\, dr\,d\theta= n \widetilde{\beta_{\xi}} \widetilde{V}_{\widetilde{\mathbb{G}}}(yL);   \label{def-psi-3333-dec}
  \end{align} and its local version: for each $\theta\in \s$, 
   \begin{align}\label{lacal-22-3333-dec}
     \int_ {y\rho_L(\theta)}^{\infty}  h\left(\xi \left[1 - \frac{y\rho_L(\theta)}{r}\right]\right)G(r, \theta)\, dr \leq  \widetilde{\beta_{\xi}} \widetilde{\mathbb{G}}(y\rho_L(\theta), \theta).
  \end{align}
  The proof is almost identical to that of Theorem \ref{ber_type}, and hence will be omitted. 

{\bf Funding:} The first author was supported in part by the U.S. NSF Grant DMS-2000304 and by the Chateaubriand Scholarship offered by the French embassy in the United States. The fourth author was supported in part by NSERC. Additionally, part of this work was completed while the first three authors were in residence at the Institute for Computational and Experimental Research in Mathematics in Providence, RI, during the Harmonic Analysis and Convexity program; this residency was supported by the National Science Foundation under Grant DMS-1929284. The third author was additionally supported by the Simons Foundation under Grant No. 815891 for this stay.

\end{document}